\documentclass{article}

\usepackage[english]{babel}
\usepackage{tikz}
\usepackage{tikz-3dplot}
\usepackage[letterpaper,top=2cm,bottom=2cm,left=3cm,right=3cm,marginparwidth=1.75cm]{geometry}
\usepackage{ulem}
\usepackage[utf8]{inputenc}

\usepackage{amssymb,amsfonts}
\usepackage{enumitem} 
\usepackage{amsmath} 
\usepackage{hyperref}
\usepackage{prettyref}

\usepackage{amsthm}
\usepackage{fullpage}
\usepackage{graphicx}
\usepackage{xcolor}
\usepackage{comment}
\usepackage{subcaption}
\usepackage{todonotes}
\usepackage{color}
\usepackage{authblk}
\usepackage{ulem}
\usepackage{epstopdf}
\usepackage{algorithm}
\usepackage{algorithmic}

\DeclareMathOperator*{\argmin}{arg\,min}

\title{Spectra of Multivariate Vandermonde Matrices and Min-max Error Bounds for the High-dimensional Super-resolution Problem}
\author[1]{Nuha Diab}
\author[2]{Dmitry Batenkov}
\affil[1]{Department of Applied Mathematics, Tel Aviv University, Tel Aviv, Israel}
\affil[2]{Basis Research Institute, NYC, USA}

\newcommand{\ci}{{\cal X}} 
\newcommand{\qq}{{\cal Q}} 
\newcommand{\Gg}{{\cal G}} 
\newcommand{\Zz}{{\cal Z}}
\newcommand{\Tt}{{\cal T}}
\newcommand{\Rr}{{\cal R}}
\newcommand{\EE}{{\cal E}}

\newcommand{\td}[1][d]{\mathbb{T}^{#1}} 

\newcommand{\uu}{\boldsymbol{u}} 
 
\newcommand{\Ll}{\mathcal{L}}

\newcommand{\ZZ}{\mathbb{Z}}
\newcommand{\NN}{\mathbb{N}}

\newcommand{\RR}{\mathbb{R}}
\newcommand{\TT}{\mathbb{T}}
\newcommand{\CC}{\mathbb{C}}

\newcommand{\srf}{\text{SRF}}

\DeclareMathOperator\supp{supp}
\DeclareMathOperator{\rank}{rank} 

\DeclareMathOperator{\diag}{diag}

\newtheorem{theorem}{Theorem}
\newtheorem{lemma}{Lemma}
\newtheorem{definition}{Definition}
\newtheorem{proposition}{Proposition}
\newtheorem{remark}{Remark}
\newtheorem{corollary}{Corollary}
\begin{document}
\maketitle

\begin{abstract}
We study multivariate rectangular Vandermonde matrices \(U_N\) whose rows are indexed by the grid
\(
\Gg_N=\{-N,\ldots,N\}^d,
\)
and whose columns correspond to irregularly spaced points on the unit $d$-dimensional torus. Motivated by the problem of super-resolution (recovery of spike trains from low frequency noisy measurements), we consider configurations in which the points form several clusters, with points within each cluster separated by at most \(h\) and distinct clusters are well-separated. We show that the column subspaces associated with different clusters are nearly orthogonal, i.e. for $N \rightarrow \infty, Nh \ll 1$, the minimal principal angle between any such pair is at least
\[
    {\pi \over 2} - \frac{C_1}{N} - C_2Nh - \frac{C_3}{\sqrt{N}},
\]
for some constants $C_1,C_2,C_3$ independent of $N$ and $h$.
This near-orthogonality enables the spectral analysis of such \(U_N\) to be reduced to the analysis of each individual cluster.

As a consequence, we derive sharp scaling for all singular values of \(U_N\), revealing how the geometry of clustered points controls the conditioning of the high-dimensional super-resolution problems. Then we establish min-max error bounds for the recovery of clustered point sources, requiring only an upper bound on the largest cluster size and no knowledge of the
geometry within the clusters. These bounds coincide with the optimal rates known in the one-dimensional setting. Furthermore, when the clustered points are in generic position (i.e. do not lie on a low-dimensional algebraic variety) our results indicate that super-resolution in high dimensions is more stable than in one dimension.
\end{abstract}

\section{Introduction}\label{sec:intro}

Let $\TT^d := (-\pi,\pi]^d\equiv \left(\RR\mod 2\pi\right)^d$, let $\ci := \{x_j\}_{j=1}^n \subset \TT^d$ be a set of $n$ distinct nodes, and let
\begin{equation}\label{eq:grid}
    \Gg_N := \{-N,\dots,-1,0,1,\dots,N\}^d \subset \ZZ^d
\end{equation}
denote the frequency grid of bandwidth $N$. The main object we study in this paper is the Vandermonde matrix of the nodes (i.e.\ the high-dimensional Vandermonde matrix), defined as follows
\begin{equation} \label{eq:multi-van}
    U(\ci;\Gg_N) := \big[\exp{\imath \langle \omega, x_j \rangle}\big]_{\omega \in \Gg_N}^{j=1,\dots,n}\in\CC^{|\Gg_N|\times n},
    \qquad \langle \omega, x_j \rangle := \sum_{k=1}^d \omega_k x_{j,k}.
\end{equation}
Set $U_N := U(\ci;\Gg_N)$ whenever the node set is clear from the context and let $\Delta$ be the minimal separation between the nodes. In the regime $N\Delta \ll 1$, our goal is to determine the scaling of all the singular values of $U_N$, together with their multiplicities, in terms of the geometry of the node set $\ci$. The main application, developed in the second part of this introduction and in Section~\ref{sec:min-max}, is a min-max error bound for the high-dimensional super-resolution problem.

\par Vandermonde matrices arise naturally in a wide range of problems in numerical
analysis, including polynomial interpolation, approximation theory, nonuniform discrete Fourier transforms, and quadrature
rules, and have received sustained attention in the literature
\cite{beckermann2000condition, beckermann2019bounds, pan2016bad, gautschi1962inverses,
Gautschi1963, barnett2022,wilber2025}. They also appear prominently in signal processing, particularly in
spectral estimation and direction-of-arrival problems, where they are closely tied
to data covariance and sampling operators \cite{lee1992eigenvalues, stoica2005spectral,
tuncer2009classical}. Within the applied harmonic analysis community, Vandermonde matrices have more recently attracted interest through their role in mathematical super-resolution \cite{batenkov2018stability, batenkov2021a, demanet2015, Vandermonde, kunis2020a, diederichs2019well,fei2024towards}, where the smallest singular value of $U_N$ controls both the optimal reconstruction bounds (i.e. min-max error bounds) and the robustness of subspace methods \cite{li2019conditioning,li2020a}; we return to this application below.

\par The spectrum of $U_N$ is determined by the geometry of $\ci$ relative to the scale $\frac{1}{N}$, the Rayleigh--Nyquist limit. Let $\Delta$ denote the minimal separation between any two nodes. When the nodes are well separated, i.e.\ $\Delta \gtrsim \frac{1}{N}$, the matrix $N^{-\frac{d}{2}}U_N$ is uniformly well conditioned \cite{kunis2021,aubel2019vandermonde}. The difficult regime is the opposite one, in which some of the nodes are much closer to each other than $\frac{1}{N}$ and for example form clusters leading to ill-conditioning of the matrix $U_N$.

\par The complete spectral properties of $U_N$ (e.g.\ the asymptotic scaling of all the singular values as $N\Delta \rightharpoondown 0$, approximations to the singular vectors) are of both theoretical and practical interest. For instance, the number of singular values above a noise threshold would count the degrees of freedom in a signal, and the corresponding signal subspace would provide a stable projection basis which can be useful for recovery algorithms (consider for instance the SVD thresholding in classical inverse problems \cite{hansen2010discrete}). Moreover, singular values of sampling matrices (such as Hankel and Toeplitz matrices) which contain Vandermonde matrices in their decomposition, are used for detecting admissible sampling rates for the univariate super-resolution problem \cite{diab2025selecting}.

\par The spectrum of the matrix $U_N$ in this regime has been investigated in the one-dimensional case \cite{batenkov2021}. The main assumption is that the nodes $\{x_j\}$ are divided into disjoint clusters $\mathcal{C}_i$, each of which is contained in an interval of length at most $\Delta \leq {1\over N}$, while the inter-cluster distances are at least $\eta \geq {1\over N}$. Their main result establishes that the subspaces of $\mathbb{C}^{N+1}$ corresponding to each cluster (the so-called ``cluster subspaces'') are nearly orthogonal. As a result, spectral analysis of $U_N$ is significantly simplified, reducing the problem to the analysis of each cluster separately. Thus the singular values of $U_N$ can be expressed as a scaled union of the singular values of each $U_N(\mathcal{C}_i)$. The derivation of the singular values is asymptotic only. In more detail, when the size of all the clusters is of the same order $\Delta$, the singular values of $N^{-{1\over 2}}U_N$ have the following simple scales (up to constants):
\begin{align*}
    {(N\Delta)}^0, {(N\Delta)}^1, \dots, {(N\Delta)}^{\ell -1},
\end{align*}
where $\ell$ denotes the number of nodes in the largest cluster. Furthermore, the number of singular values scaling as ${(N\Delta)}^{j -1}$ is exactly equal to the number of clusters with at least $j$ nodes.

\par In dimension $d>1$ the picture changes, since the scaling of the singular values is no longer determined by the number of nodes in each cluster alone, but also by the geometry of the nodes within a cluster (i.e.\ the algebraic variety the points lie on). In our recent work \cite{diab2024spectral}, we find the scaling of all the singular values of $U_N$ when the nodes form only a single cluster. In more detail, for a single cluster of size of order $\Delta$, the singular values of $N^{-{d\over 2}}U_N$ split into $m+1$ groups with the following scalings:
\[
   {(N\Delta)}^{k}, \quad k=0,\dots,m,
\]
where the number of elements in each group is determined explicitly by the geometry of the cluster. Here $m=\mu(\ci)$ is the \emph{discrete moment order} of the cluster (Definition~\ref{def:mu-def}): the smallest degree at which
polynomials interpolate the nodes, i.e.\ the least $m$ such that for any set of values $v_1,\dots,v_n \subset \mathbb{C}$, there exists a polynomial $p$ of degree at most $m$ such that $p(x_j)=v_j$ for $j=1,\dots,n$.

For instance, three nodes in $\RR^2$ forming a triangle are interpolated by an affine function
$a+bx_1+cx_2$, so $m=1$ and only two scales appear, $(Nh)^0$ and $Nh$, the latter twice; three collinear nodes reduce to the univariate problem at three points and need a quadratic, so $m=2$ and the three scales $(Nh)^0$, $Nh$, ${(Nh)}^{2}$ each appear once.

\par In this work, we extend the near-orthogonality technique introduced in
\cite{batenkov2021} to the high-dimensional case, with the goal of determining the
scaling of the singular values of $U(\ci;\Gg_N)$ in terms of the geometry of the
clusters formed by the node set $\ci$. Establishing near-orthogonality in this
setting requires applying a new ``projection'' technique to the multivariate
Vandermonde matrix \eqref{eq:multi-van}, described in detail in the next section.
This yields a complete description of the spectrum for several clusters with
arbitrary geometries (Theorem~\ref{thm:entire-spec}): when the sizes of all the
clusters are of the same order $\Delta$, the singular values of
$N^{-{d\over 2}}U_N$ have the scales (up to constants)
\begin{align*}
    {(N\Delta)}^{0}, {(N\Delta)}^{1}, \dots, {(N\Delta)}^{m},
\end{align*}
where $m:=\max_i \mu(\mathcal{C}_i)$ and the multiplicity of each scale is
obtained by summing the corresponding contributions of the individual clusters. In
particular, the smallest singular value of $U(\ci;\Gg_N)$ scales as
${(N\Delta)}^{m}$.

\par One of the first results on the scaling of the smallest singular value of $U_N$ appears in \cite{kunis2020a}. There, the authors establish a lower bound that scales like $(N\Delta)^{\ell-1}$, with explicit constants independent of $N$ and $\Delta$, where $\ell$ denotes the number of nodes in the largest cluster. They also obtain matching upper bounds for the case $\ell = 2$ and $d \geq 1$: in this case, the nodes within a cluster lie on a line, which represents the worst-case geometry and reduces the problem to the one-dimensional setting. They further provide examples of node configurations in $\RR^2$ in different geometries for which the lower bound is not sharp and depends on the geometry within the clusters.

\par A recent improvement on this result is established in \cite{li2024a}. There, the lower bounds are also established for nodes in different geometries (where the node don't necessarily form clusters). In more detail, if the nodes lie near the union of $r$ hyperplanes, the lower bound would scale like $(N\delta)^r$, where $\delta$ is the maximal distance between any node and the hyperplanes. We emphasize that this result considers the geometry of all the nodes (and not only the geometry of cluster nodes as in our case). Moreover, for nodes in general position, the provided lower bound scales like $(N\Delta)^{\lceil{\ell - 1 \over d}\rceil}$. Comparing the latter result to ours, we provide both lower and upper bounds that
scale like $(N\Delta)^m$; in general position $m$ takes its smallest value, namely
the smallest integer with ${d+m-1\choose d}<\ell\leq {d+m\choose d}$ (see
Definition~\ref{def:mu-def} and the discussion following it). For $d > 1$, we have
$m < {\lceil{\ell - 1 \over d}\rceil}$ (and equality holds for $d=1$), but in our
result we assume the clusters are well-separated and in \cite{li2024a} there is no
such assumption.

\par We now turn to the main application of the above, the super-resolution problem. Let $\delta(\cdot)$ denote the Dirac delta distribution, and consider the signal
\begin{equation} \label{eq:spike}
    f(t) := \sum_{j=1}^n\alpha_j\delta (t-x_j), \quad t, x_j \in \TT^d, \quad \alpha_j \in \CC,
\end{equation}
to be recovered. We refer to $\{x_j\}$ as the ``nodes'' and to $\{\alpha_j\}$ as the ``coefficients'' or ``amplitudes''. The noisy measurements are given by
\begin{equation} \label{eq:samples}
    \hat{f}(\omega) := \sum_{j=1}^n \alpha_j e^{\imath \langle \omega,x_j \rangle}+\epsilon(\omega), \quad \omega \in \Gg_N,
\end{equation}
where $|\epsilon(\omega)|\leq \epsilon$. In matrix form these measurements read $\hat f = U(\ci;\Gg_N)\alpha + \epsilon$, which is precisely why the spectrum of $U_N$ governs the stability of the recovery. This problem appears in a wide range of applications, including direction-of-arrival (DOA) estimation \cite{stoica2005spectral}, radar and inverse scattering \cite{cuyt2020sparse}, fluorescence microscopy and single-molecule imaging \cite{bhandari2016signal}, and more generally in multidimensional spectral estimation and imaging \cite{lindberg2012,bertero1996}.

\par When the minimum separation $\Delta$ between any two spikes (or nodes) is less than $\frac{1}{N}$, the inverse problem becomes ill-conditioned and stability is not guaranteed~\cite{donoho1992a}. For $d=1$, this occurs because applying the inverse Fourier transform to band-limited data will not resolve two nodes that are closer than $\frac{1}{N}$. This regime is known as the super-resolution regime and is characterized by the \textit{super-resolution factor}
\begin{equation}\label{def:srf}
   \srf := \frac{1/N}{\Delta} = \frac{1}{N\Delta}
\end{equation}
being larger than $1$, while the well-separated case corresponds to $\srf \leq 1$. Extensive research has been conducted on one-dimensional scenarios, addressing both well-separated and super-resolution regimes \cite{fernandez2016super,duval2017sparse,candes2014,Li&Liao2020,cuyt2018,petz2021,batenkov2021a,Vandermonde}. Also in the high-dimensional setting, the $\srf$ is defined as in \eqref{def:srf}. This matches the $1$D case since the sampling kernel associated with $\Gg_N$ factorizes into one-dimensional Dirichlet kernels, yielding a resolution limit of order $1/N$ in each coordinate.

\par A major challenge in the field is to determine resolution limits based on
prior knowledge about the unknown signal and the limitations of the measurement
system  \cite{batenkov2021a}. To address this rigorously, we examine the so called min-max error bounds,
which characterize the worst-case performance of the best possible recovery
algorithm. These error bounds provide theoretical guarantees and when an algorithm
achieves them in some sense (for example asymptotically), it is considered optimal.
The importance of these bounds is not only in proving the optimality of methods but
also in understanding how the prior knowledge affects the stability and accuracy of
certain reconstructions.

\par Using the scaling of the smallest singular value of $U_N$ derived in
Theorem~\ref{thm:entire-spec}, we establish min-max error bounds for nodes
restricted to a $d$-dimensional grid with spacing $\Delta$ (referred to as ``on-grid'' min-max error bounds). Consider signals of the form \eqref{eq:spike} whose nodes form multiple clusters, with an upper bound $\ell$ on the largest cluster size. We prove in
Theorem~\ref{theorem:min-max-error} that the min-max error bounds scale like ${\srf}^{2\ell - 1}\epsilon$.
The exponent $2\ell-1$ exactly matches the univariate case; indeed, the lower bounds are achieved by restricting the nodes to be collinear. A consequence worth emphasizing is that these bounds are pessimistic for generic configurations. When the clustered nodes are randomly positioned, and thus do not lie on any low-dimensional algebraic variety, our result on the smallest singular value suggests that super-resolution in high dimensions is in fact more stable than in 1D; see the discussion in Section~\ref{subsec:min-max-random}.

\par Similar ``off-grid'' min-max error bounds were derived in \cite{liu2024improved}
(which is an improvement to \cite{liu2021m}), where the nodes are not restricted to a grid, form a single cluster only, and
the cluster is assumed to be centered around zero. In our result, we generalize the
latter and assume multiple well-separated clusters centered around a given set of
points (named reference points, see Definition~\ref{def:config-pro}). The price is
that our nodes are restricted to the grid with spacing $
\Delta$; establishing ``off-grid'' bounds appears to require multidimensional ``confluent'' Vandermonde matrices,
cf.\ \cite{batenkov2013algebraic, batenkov2013c, batenkov2018stability,
batenkov2023}, and we leave that to future work.

\par This paper is organized as follows. First, in Sections~\ref{sec:pre} and~\ref{sec:proof-outline}, we present the definitions and notations alongside an outline of the proof of Theorem~\ref{thm:entire-spec} (i.e., the scaling of all singular values of $U_N$). Then, the proof itself is presented in Sections~\ref{sec:ex-good-proj}, \ref{sec:van-row-par}, \ref{sec:cl-sub-orth} and \ref{spec-dis}.  
Finally, in Section~\ref{sec:min-max}, we derive the min-max error bounds and provide numerical results showing that the 2D-ESPRIT \cite{ESPRIT1990, sahnoun2017} attains the optimal bounds, in the case of $d=2$, without any prior knowledge of the cluster geometries. 

\section{Preliminaries}\label{sec:pre}
We recall some definitions and notations from $\cite{batenkov2021}$ and generalize them for the high-dimensional setting ($d \geq 1$).
\begin{definition}\label{def:wrap_dist}
    For $x,y \in \TT^d$, we denote the wrap-around distance by
    \[ \|x-y\|_{\TT^d}:=\min_{r \in (2\pi \ZZ)^d}\|x-y+r\|_{\infty}. \]
\end{definition}

\begin{definition}\label{def:single_cluster}
    The node set $\mathcal{X}=\{x_1,...,x_n\} \subset \TT^d$ is said to form 
    \begin{itemize}
        \item an $(h,n)$-cluster if
        \[ \forall x,y \in \mathcal{X}, x \neq y: \quad 0 \leq \|x-y\|_{\TT^d} \leq h;\]
        \item an $(h,\tau,n)$-cluster, for some $\tau>0$, if
        \[ \forall x,y \in \mathcal{X}, x \neq y: \quad \tau h \leq \|x-y\|_{\TT^d} \leq h.\]
    \end{itemize}
\end{definition}

\begin{definition}\label{def:multi_cluster}
    The node set $\mathcal{X}=\{x_1,...,x_n\} \subset \TT^d$ is said to form an $((h^{(j)},n^{(j)})_{j=1}^M),\eta)$ (respectively, $((h^{(j)},\tau^{(j)},n^{(j)})_{j=1}^M),\eta)$)-clustered configuration if there exists an $M$-partition $\ci = \bigsqcup_{j=1}^{M} \mathcal{C}^{(j)}$, such that for each $j\in \{1,\dots,M\}$ the following conditions are satisfied:
    \begin{itemize}
        \item $\mathcal{C}^{(j)}$ is an $(h^{(j)},n^{(j)})$ (respectively, an $(h^{(j)},\tau^{(j)},n^{(j)})$)-cluster.
        \item $\|x-y\|_{\TT^d} \geq \eta > 0. \ \forall x \in \mathcal{C}^{(j)}, \forall y \in \ci \setminus\mathcal{C}^{(j)}$.
    \end{itemize}
\end{definition}
We denote by $\|\cdot\| := \|\cdot\|_2$ thorough the paper.

\begin{definition}
    For any $p \in \mathbb{N} \cup \{0\}$, we write $K \asymp N^{p}$ if there exist constants 
    $0 < C_1 \leq C_2$, independent of $N$, such that
    \[
        C_1 N^{p} \leq K \leq C_2 N^{p}.
    \]
\end{definition}

In addition, for any set of nodes $\mathcal{Y}\subset\TT^d$, we denote by $U_N(\mathcal{Y}) := U(\mathcal{Y},
\Gg_N)$ where $U(\mathcal{Y},\Gg_N)$ is defined in \eqref{eq:multi-van}.

\par We now recall some definitions from \cite{diab2024spectral}.

For a multi-index
\(
\boldsymbol{\alpha}=(\alpha_1,\ldots,\alpha_d)\in\mathbb{Z}_+^d,
\)
we define its total degree and factorial by
\[
|\boldsymbol{\alpha}|
=
\sum_{j=1}^{d}\alpha_j,
\qquad
\boldsymbol{\alpha}!
=
\alpha_1!\cdots\alpha_d!.
\]

For every nonnegative integer \(j\), let
\[
\mathbb{P}_j
:=
\{\boldsymbol{\alpha}\in\mathbb{Z}_+^d:\ |\boldsymbol{\alpha}|\le j\},
\]
denote the set of multi-indices of degree at most \(j\), and let
\[
\mathbb{H}_j
:=
\{\boldsymbol{\alpha}\in\mathbb{Z}_+^d:\ |\boldsymbol{\alpha}|=j\}
=
\mathbb{P}_j\setminus\mathbb{P}_{j-1},
\]
with the convention \(\mathbb{P}_{-1}=\emptyset\). Their cardinalities are
\begin{equation}\label{eq:card}
p_j:=|\mathbb{P}_j|
=
{j+d\choose d},
\qquad
h_j:=|\mathbb{H}_j|
=
{j+d-1\choose d-1}
=
p_j-p_{j-1}.
\end{equation}

Given \(\boldsymbol{x}=(x_1,\ldots,x_d)\in\mathbb{R}^d\) and a multi-index
\(\boldsymbol{\alpha}\in\mathbb{Z}_+^d\), we use the standard notation
\[
\boldsymbol{x}^{\boldsymbol{\alpha}}
=
x_1^{\alpha_1}\cdots x_d^{\alpha_d},
\]
where the degree of the monomial \(\boldsymbol{x}^{\boldsymbol{\alpha}}\) is \(|\boldsymbol{\alpha}|\). Finally, we fix a graded monomial ordering \(\prec\) on
\(\mathbb{Z}_+^d\), such as the graded lexicographic or graded reverse
lexicographic order.

\begin{definition}[Multivariate Vandermonde matrix]
Let
\(
\mathcal{X}
=
\{\boldsymbol{x}_1,\ldots,\boldsymbol{x}_n\}
\subset \mathbb{R}^d
\)
be a finite set of nodes and let
\(
\mathcal{A}
=
\{\boldsymbol{\alpha}_1,\ldots,\boldsymbol{\alpha}_m\}
\subset\mathbb{Z}_+^d
\)
be ordered according to \(\prec\). The corresponding multivariate
Vandermonde matrix is
\begin{equation}\label{eq:multiVan}
V_{\mathcal{A}}(\mathcal{X})
=
\Big[
(\boldsymbol{x}_i)^{\boldsymbol{\alpha}_j}
\Big]_
{\substack{1\le i\le n\\1\le j\le m}}
\in\mathbb{R}^{n\times m}.
\end{equation}
\end{definition}

It is important to distinguish this matrix from the Fourier-type
Vandermonde matrix \(U_N\) introduced in \eqref{eq:multi-van}. While
\(V_{\mathcal{A}}(\mathcal{X})\) is built from monomials,
the entries of \(U_N\) are complex exponentials of the form
\(e^{i\langle \omega,x\rangle}\).

In addition, we write
\[
V_{\le k}:=V_{\mathbb{P}_k}(\mathcal{X}),
\qquad
V_k:=V_{\mathbb{H}_k}(\mathcal{X}).
\]
Since \(\mathbb{P}_k\) decomposes according to total degree,
\(V_{\le k}\in\mathbb{R}^{n\times p_k}\) admits the block representation
\begin{equation}\label{eq:van-deg-k}
V_{\le k}
=
\big[
V_0 \;\; V_1 \;\; \cdots \;\; V_k
\big],
\end{equation}
where \(V_j\in\mathbb{R}^{n\times h_j}\) contains all monomials of
degree exactly \(j\).

The above notation follows the framework developed in
\cite{usevich2021}. We next introduce the discrete moment order, a quantity that captures the geometry of the node set \(\mathcal{X}\).

\begin{definition}[Discrete moment order {\cite{schabak2005}}]
\label{def:mu-def}
Let \(\mathcal{X}\) be a set of \(n\) distinct points. The
\emph{discrete moment order} of \(\mathcal{X}\), denoted by
\(\mu(\mathcal{X})\), is the smallest integer \(m\) such that
\[
\operatorname{rank}\bigl(V_{\le m}(\mathcal{X})\bigr)=n.
\]
Equivalently,
\[
\ker\bigl(V_{\le m}(\mathcal{X})^T\bigr)=\{0\}.
\]
Moreover, \(\mu(\mathcal{X})\le n-1\) (for more details see \cite{schabak2005,GascaMariano2000,wathen2015spectral}). 
\end{definition}

As mentioned in the introduction, $\mu(\mathcal{X})$ is equivalently the smallest
degree at which polynomials interpolate $\mathcal{X}$: the least $m$ such that any
values assigned to the points of $\mathcal{X}$ are matched by some polynomial of
degree at most $m$. 

This measure depends both on the number of nodes and on the underlying algebraic
variety on which the nodes are located. For example, consider the two sets
$\mathcal{X}_p = \{(x_i,x_i^2)\}_{i=1}^6$ and
$\mathcal{X}_c = \{(x_i,x_i^3)\}_{i=1}^6$ in $\RR^2$, both with the same number of
nodes; their discrete moment orders are $\mu(\mathcal{X}_p)=3$ and
$\mu(\mathcal{X}_c)=2$. Indeed, restricting a polynomial of degree at most $m$ to
the parabola leaves exactly the univariate polynomials of degree at most $2m$, so
interpolating six nodes requires $2m+1\ge 6$; on the cubic, the monomials
$x^ay^b$ with $a+b\le 2$ restrict to the six distinct powers
$x^0,x^1,x^2,x^3,x^4,x^6$, which already suffice. Note that the nodes lying on the
curve of higher degree have the smaller moment order. In general, for a
fixed number $n$ of nodes in $\RR^d$, the discrete moment order is at least $k$,
where $k$ is the smallest integer such that ${d+k-1\choose d}< n \leq
{d+k\choose d}$, and at most $n-1$, which occurs when all the nodes lie on a
single line.
\par We conclude this section by recalling the main result of
\cite{diab2024spectral}, which characterizes the scaling of the eigenvalues of
\(U_N(\mathcal{X})^*U_N(\mathcal{X})\) when \(\mathcal{X}\) forms a
single cluster. For our purposes, we state the result in terms of the
singular values of \(U_N(\mathcal{X})\).

\begin{theorem}\label{thm:sr}
    For $\ci$ forming a $(h,\tau,n)$-cluster, let $m=\mu(\ci)$ be the discrete moment order of $\ci$.
    For every $\epsilon_0>0,\alpha\in(0,1)$ there exists $N_0$ such that for all $N\geq N_0$ and $\alpha\epsilon_0 < N\tau h < \epsilon_0$,    
    the eigenvalues of $U_N:=U_{N}(\ci)$ split into $m+1$ groups
\[
\sigma_{0,0}(U_N) \asymp N^{d\over 2}{(Nh)}^{0}, \quad \{\sigma_{1,j}(U_N)\}_{j=1}^{t_1}\asymp N^{d\over 2}{(Nh)}^1, ... ,\{\sigma_{m,j}(U_N)\}_{j=1}^{t_m}\asymp N^{d\over 2}{(Nh)}^{m},
\]
where $t_k:=\rank(V_{\leq k}(\ci))-\rank(V_{\leq k-1}(\ci))$, $k=1,\dots,m$. 
\end{theorem}

\section{Main result and proof outline}\label{sec:proof-outline}
Let $\ci$ be a $((h^{(j)},\tau^{(j)},n^{(j)})_{j=1}^M),\eta)$-clustered configuration as in definition $\ref{def:multi_cluster}$. Our main approach to finding the scaling of the singular values of $U_N(\ci)$ is proving near-orthogonality between column spaces spanned by each cluster (for certain separation between clusters), as in \cite{batenkov2021}. In the latter, it was shown that for nodes in 1D, if the clusters are well-separated, then the singular values of $U_N(\ci)$ are a scaled union of all the singular values of each sub-matrix $U_N(\mathcal{C}_i)$. In this section, we  generalize this result to the high-dimensional setting.

\par The core idea is simple: if the clusters are well-separated in a $d$-dimensional space, then there exists a direction $w$, where the projected clusters onto a line with direction $w$, are still well-separated (see figure \ref{fig:good-proj} for an illustration).
 \begin{figure}
        \centering
        \includegraphics[width=0.7\linewidth]{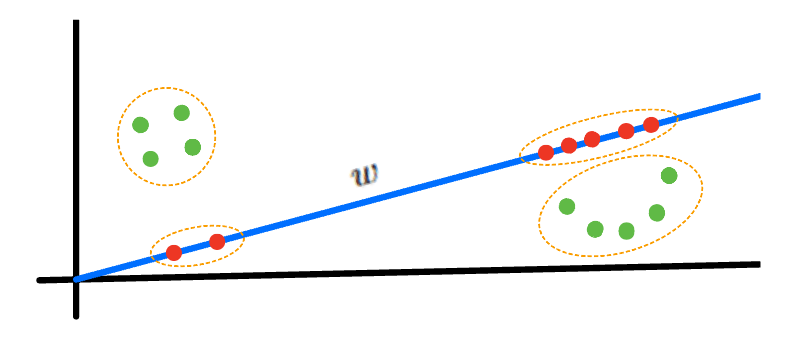}
        \caption{Projecting clusters into the blue line with direction $w$ preserves the well-separation between the clusters. In this case we say that $w$ is an admissible projection.}
        \label{fig:good-proj}
    \end{figure}
Now, since the projected clusters are well-separated, then the subspaces spanned by each projected cluster are near-orthogonal. Building on this, we prove that the subspaces spanned by each original cluster are also near-orthogonal. Combined with our results in \cite{diab2024spectral} on the scaling of the singular values of $U_N$ with nodes in a single cluster, we achieve our desired result.
\par For the reader's convenience, we give a detailed step by step outline of the proof:
\begin{enumerate}
    \item For each pair of clusters, find $w \in \Gg_N\setminus \{\boldsymbol{0}\}$ (independent of $N$) such that the projected clusters onto the line with direction $w$ are well separated (see Proposition \ref{prop:proj} in Section \ref{sec:ex-good-proj}). We refer to such direction as an admissible projection. The geometry of nodes within each cluster under this projection is of no importance (see remark \ref{rem:inner-cluster-proj}).
    \item For a cluster $\mathcal{C}$ and $K_2\asymp N$, let $V_{w}(\mathcal{C}) := \big[e^{\imath k\langle w, x\rangle}\big]_{k=1,...,K_2}^{x\in \mathcal{C}}$ be the univariate Vandermomde matrix of the projected nodes corresponding to an admissible projection $w$.
    In Theorem \ref{thm:van-part} of Section \ref{sec:van-row-par}, we find shifts $\{s_i\}_{i=1}^{K_1} \subset \Gg_N$ such that the matrix $U_N(\mathcal{C})$ can be rearranged (up to a permutation) as follows:  
    
    \[
        U_N(\mathcal{C}) = \begin{bmatrix}
            V_w(\mathcal{C}) \\
            V_w(\mathcal{C})D_1(\mathcal{C}) \\
            \vdots \\
            V_w(\mathcal{C})D_{K_1}(\mathcal{C}) \\
            R(\mathcal{C})
        \end{bmatrix}
    \]
    where $D_k(\mathcal{C}) := diag(e^{\imath \langle s_k, x\rangle})_{x\in \mathcal{C}}$ and $K_1 \asymp N^{d-1}$.
    
    \item For any two clusters $\mathcal{C}_j,\mathcal{C}_k$ in $\ci$, we prove in Theorem \ref{thm:orth} of Section \ref{sec:cl-sub-orth} that if the subspaces that are spanned by $V_w(\mathcal{C}_k)$ and $V_w(\mathcal{C}_j)$ are near-orthogonal (which they are by 1) above) , then the subspaces that are spanned by
    \[
        \begin{bmatrix}
            V_w(\mathcal{C}_k) \\
            V_w(\mathcal{C}_k)D_1(\mathcal{C}_k) \\
            \vdots \\
            V_w(\mathcal{C}_k)D_{K}(\mathcal{C}_k) \\
            R(\mathcal{C}_k)
        \end{bmatrix}, \quad \textbf{and} \qquad
        \begin{bmatrix}
            V_w(\mathcal{C}_j) \\
            V_w(\mathcal{C}_j)D_1(\mathcal{C}_j) \\
            \vdots \\
            V_w(\mathcal{C}_j)D_{K}(\mathcal{C}_j) \\
            R(\mathcal{C}_j)
        \end{bmatrix}
    \]
    are also near-orthogonal.
    \item Finally, using Theorem 5.1 in \cite{batenkov2021}, the singular values of $U_N(\ci)$ are a scaled union of the singular values of each $U_N(\mathcal{C}_j)$ which are found in \prettyref{thm:sr}. 
\end{enumerate}

We state the main result of this paper.
\begin{theorem}[Entire Spectrum]\label{thm:entire-spec}
    Let $\ci$ form an $((h^{(i)},\tau^{(i)},n^{(i)})_{i=1}^M,\eta)$-clustered configuration where ${\pi \over a} \leq \eta \leq 2\pi$ for $a=\left\lfloor {\sqrt[d-1]{4(\sqrt{d}/2+1)+1}\over 2} \right\rfloor + 1$. Furthermore, set $\tau = \min_{i}\{\tau^{(i)}\}$ and suppose that $h^{(1)}=\dots=h^{(M)}$. For each cluster $\mathcal{C}_i$ let $m_i = \mu(\mathcal{C}_i)$ be its discrete moment order. Define
    \[
        \ell_j := \sum_{i=1}^M t_j^{(i)}, \ j=0,...,\max_i \{m_i\},
    \] 
    where $t_j^{(i)}:=\rank(V_{\leq j}(\mathcal{C}_i))-\rank(V_{\leq j-1}(\mathcal{C}_i))$.
    Then, for constants $C_6,C_7$ and for every $ 0 < \epsilon_0 < C_7, \alpha \in (0,1)$ there exists $N_0$ such that for all $N \geq N_0$ and $\max\{{\alpha\epsilon_0 \over \tau h}, {C_6}\} < N < {\epsilon_0 \over \tau h}$, there are precisely $\ell_j$ singular values of $U_N(\ci)$ scaling like $N^{d\over 2}(Nh)^{j}$.
\end{theorem}
\begin{proof}
    Using Theorem \ref{thm:sr} for each cluster $\mathcal{C}_i$, we have that for every $\epsilon_0 > 0, \alpha \in (0, 1)$ there exists $N_i$ such that for all $N \geq N_i$ and $\alpha\epsilon_0 < N(\tau^{(i)}h) < \epsilon_0$, the singular values of $U_i := {1\over \sqrt{(2N+1)^d}}U_N(\mathcal{C}_i)$ split into $m + 1$ groups
    \[
\sigma_{0,0}(U_i) \asymp {(Nh)}^{0}, \quad \{\sigma_{1,j}(U_i)\}_{j=1}^{t_1^{(i)}}\asymp{(Nh)}^1, ... ,\{\sigma_{m,j}(U_i)\}_{j=1}^{t_m^{(i)}}\asymp{(Nh)}^{m},
\]
where $t_k^{(i)}:=\rank(V_{\leq k}(\mathcal{C}_i))-\rank(V_{\leq k-1}(\mathcal{C}_i))$, $k=1,\dots,m_i$. Set $N_0 := \max\{N_i\}_{i=1}^M$. The cluster $\mathcal{C}_i$ is also a $(h^{(i)},\tau,n^{(i)})$ cluster. Thus we have 
\begin{align*}
    &\alpha\epsilon_0 \leq N\tau h \leq \epsilon_0 \ \iff \ {\alpha\epsilon_0 \over \tau h} < N < {\epsilon_0 \over \tau h}.
\end{align*}

Combining the above with Theorem \ref{thm:multi-sv}, we get the desired result.
\end{proof}
\begin{remark}
    The condition $Nh > {\alpha\epsilon_0\over \tau}$ for $\alpha \in (0, 1)$ is imposed by Theorem \ref{thm:sr} only.
\end{remark}

\section{Existence of admissible projections}\label{sec:ex-good-proj}
We denote the center of each cluster $\mathcal{C}_j$ by $c_j:={1\over n^{(j)}}\sum_{x\in \mathcal{C}_j}x$ for $j=1,...,M$.

\begin{definition}
    For any polyhedron $T \subset \RR^d$, we denote by $|T|$ the volume of $T$.
\end{definition}
\begin{definition}[Slab]
Let $u \in \mathbb{R}^d$ be a unit vector, $c \in \mathbb{R}$, and $h>0$. 
The slab with normal direction $u$, center $c$, and half-width $h$ is defined by
\[
S(u,c,h)
:=
\left\{ x \in \mathbb{R}^d : 
\big| \langle x,u\rangle - c \big| \le h
\right\}.
\]
Equivalently, it is the region between the two parallel hyperplanes 
$\langle x,u\rangle = c \pm h$, and its width equals $2h$.
\end{definition}
{In the following Theorem, we prove that for any two clusters with centers $c_j,c_k$,  the measure of "bad" projections is bounded.
\begin{theorem}\label{thm:slab}
Let $c_j,c_k\in\mathbb{R}^d$ and define $y:=c_j-c_k$ with 
$\theta:=\|y\|>0$. Fix $\alpha\in[0,\pi]$ and $a>0$, and let 
$\qq_a=[-a,a]^d$. Define
\[
\Lambda_{j,k}^{\alpha}(\qq_a)
:=\left\{w\in \qq_a :\|\langle w,c_j-c_k\rangle\|_{\mathbb{T}}\le\alpha\right\}.
\]
Then $\Lambda_{j,k}^{\alpha}(I)$ is a union of $L$ parallel slabs $I_{\ell}$
orthogonal to $y$, where
\[
L \le \left\lfloor \frac{a\theta\sqrt d}{\pi} \right\rfloor + 1.
\]
Moreover, each slab has volume bounded by
\[
|I_\ell|
\le
\frac{2\sqrt{2}\alpha}{\theta}\,(2a)^{d-1},
\qquad \ell=1,\dots,L,
\]
and consequently
\[
\big|\Lambda_{j,k}^{\alpha}(\qq_a)\big|
\le
L \cdot
\frac{2\sqrt{2}\alpha}{\theta}\,(2a)^{d-1}.
\]
\end{theorem}

\begin{proof}
Let $u:=y/\theta$ be the unit vector in the direction of $y$.
For $w\in\mathbb{R}^d$, let $s=\langle w,u\rangle$ then $\langle w,y\rangle=\theta s$.

By the definition of $\|\cdot\|_{\TT^d}$, there exists $m\in \mathbb Z$ such that
\[
\|\langle w,y\rangle\|_{\mathbb T}\le\alpha
\quad\Longleftrightarrow\quad
|\theta s-2\pi m|\le\alpha,
\]
which is equivalent to
\[
\left| s-\frac{2\pi m}{\theta} \right|
\le
\frac{\alpha}{\theta}.
\]
For each $m\in\mathbb Z$, this describes a slab orthogonal to $y$
of width $2\alpha/\theta$. Hence
\[
\Lambda_{j,k}^{\alpha}(\qq_a)
=
\bigcup_{m\in \ZZ} \left( \qq_a \cap S_m \right),
\]
where $S_m:= S(u,\frac{2\pi m}{\theta},2\alpha/\theta)$.

\par Now we count the number of slabs $S_m$ intersecting $\qq_a$.
As $w$ varies over $\qq_a=[-a,a]^d$, the scalar 
$s=\langle w,u\rangle$ ranges over an interval $I:=\{\langle w,u\rangle : \omega \in \qq_a\}$ of length at most
\[
|I| \leq \sup_{w\in \qq_a}\langle w,u\rangle
-
\inf_{ w \in \qq_a}\langle w,u\rangle
\le 2a\sqrt d,
\]
by Cauchy-Schwarz. Since consecutive slab centers are separated by
$2\pi/\theta$, the number of slabs intersecting $\qq_a$ is
\[
L
\leq
\left\lfloor
\frac{|I|}{2\pi/\theta}
\right\rfloor + 1
\le
\left\lfloor
\frac{2a\sqrt d}{2\pi/\theta}
\right\rfloor + 1
=
\left\lfloor
\frac{a\theta\sqrt d}{\pi}
\right\rfloor + 1.
\]

Finally, each intersection $I\cap S_m$ is a slice of the cube of width
$2\alpha/\theta$ in direction $u$. By Ball's cube slicing theorem \cite{ball1986cube},
every $(d-1)$-dimensional section of $[-a,a]^d$ has volume at most
$\sqrt{2}\,(2a)^{d-1}$. Therefore,
\[
|\qq_a\cap S_m|
\le
\frac{2\alpha}{\theta}\,
\sqrt{2}\,(2a)^{d-1}.
\]
Summing over all $m$ yields the desired bound.
\end{proof}

Now, using Theorem \ref{thm:slab}, we prove that for $M$ well-separated cluster centers $\{c_j\}_{j=1}^M$, there exists a "good" projection $w$ such that the projected centers remain well-separated.
\begin{theorem}\label{thm:union-bound}
    Fix $a >0$. Let $\mathcal{C} = \{c_j\}_{j=1}^M$ such that $\theta = \min_{j\neq k}\|c_j - c_k\|$ and ${\pi \over a} \leq \theta \leq 2\pi$. Then the $d$-cube $\qq_a = [-a,a]^d $ of volume $|\qq_a| = \big({2a}\big)^d$ contains a polyhedron $T \subset \qq_a$ of volume $|T| \geq {4\over M^2(\sqrt{d}+2)+2}(2a)^{d-1}$ such that for each $w \in T$ we have that for all $c_{k} \in \mathcal{C}$ and $c_j \neq c_k$:
    \begin{align*}
        \| \langle w,c_j\rangle - \langle w,c_k\rangle\|_{\mathbb{T}} \geq {1\over 2}{\sqrt{2}\pi \over (\sqrt{d}+1)M^2}.
    \end{align*}
    
\end{theorem}
\begin{proof}
    Consider the set 
    \begin{equation}
        \Lambda^{\alpha}(\qq_a) = \big\{w \in \qq_a : \exists c_k \in \mathcal{C} \ s.t. \ \min_{1\leq j \leq M}\|\langle w,c_j\rangle - \langle w,c_k\rangle\|_{\mathbb{T}} \leq \alpha\big\}
    \end{equation}
    We then have 
    \[
        \Lambda^{\alpha}(\qq_a) = \bigcup_{c_k \in \mathcal{C}}\bigcup_{c_k \neq c_j}\Lambda_{k,j}^{\alpha}(\qq_a)
    \]
 
   Using Theorem \ref{thm:slab}, we have that $\Lambda^{\alpha}(\qq_a)$ is a union of at most
   $$K = {M \choose 2}\cdot L ={ M(M-1)\over 2}\Big(\Big\lfloor\frac{a\theta\sqrt d}{\pi}\Big\rfloor + 1\Big)$$ slabs.
    \par Let $\nu$ denote the Lebesgue measure on $\mathbb{R}^d$, using that $\theta \geq {\pi \over a}$, we have
    \[
        \nu(\Lambda^{\alpha}(\qq_a)) \leq K\frac{2\sqrt{2}\alpha}{\theta}\,(2a)^{d-1} \leq \alpha\big({\sqrt{d}+1 \over \sqrt{2}\pi}M^2(2a)^d\big)
    \]
    Put $\alpha = {1\over 2}{\sqrt{2}\pi \over (\sqrt{d}+1)M^2}$, we get $\nu(\Lambda^{\alpha}(C)) \leq {1\over 2}\big({2a}\big)^{d}$.
    Now consider the complement set
    \begin{equation}
        \big(\Lambda^{\alpha}(\qq_a)\big)^c = \big\{w \in \qq_a : \forall c_k \in \mathcal{C} , \ \min_{1\leq j \leq M}\|\langle w,c_j\rangle - \langle w,c_k\rangle\|_{\mathbb{T}} > \alpha\big\}
    \end{equation}
    We have $\nu((\Lambda^{\alpha}(\qq_a))^c) \geq |C| - {1\over 2}\big({2a}\big)^{d} > {1\over 2}\big({2a}\big)^{d}$.
    \par In addition, since $\Lambda^{\alpha}(\qq_a)$ is a union of at most $K$ slabs, then $\big(\Lambda^{\alpha}(\qq_a)\big)^c$ is a union of at most $ D = K+1 \leq {M^2 \over 2}(a\sqrt{d}+1) + 1$ polyhedra. Using that $a\geq {\pi \over \theta} \geq {1\over 2}$, the average volume of these polyhedra is bounded as follows:
    \[
        {\nu((\Lambda^{\alpha}(\qq_a))^c) \over D} \geq {\big({2a}\big)^{d} \over {{M^2 \over 2}(a\sqrt{d}+1) + 1}} \geq {\big({2a}\big)^{d} \over {{M^2 \over 2}(a\sqrt{d}+2a) + 2a}} \geq {4\over M^2(\sqrt{d}+2)+2}(2a)^{d-1}.
    \]
    meaning that $\big(\Lambda^{\alpha}(\qq_a)\big)^c$ contains a polyhedron of volume greater than or equal to $${4\over M^2(\sqrt{d}+2)+2}(2a)^{d-1}.$$
\end{proof}

In order to use the "good" projection for rearranging the matrix $U_N$ (see step 2 in the proof outline), $w$ should be an element of the grid $\Gg_N$. The following Theorem proves the existence of such projection from a certain side-length size of the cube $\qq_a$.
\begin{proposition}\label{prop:proj_inG}
    For $A:=(\Lambda^{\alpha}(\qq_a))^c$ from the above Theorem, if $N > a > {\sqrt[d-1]{M^2(\sqrt{d}/2+1)+1}\over 2}$, then $\Gg_N \cap A \neq \emptyset$.
\end{proposition}
\begin{proof}
    By Theorem \ref{thm:union-bound}, the set $A$ consists of $D \leq {M^2 \over 2}(a\sqrt{d}+1) + 1$ polyhedra. Using that $a \geq {1\over 2}$ we have that $D \leq a\big({M^2 \over 2}(\sqrt{d}+2) + 2\big)$. The total volume of $A$ is at most ${1\over 2}(2a)^d$. Denote the volumes of these polyhedra by $V_1,...,V_D$. Each point in $\Gg_N$ can be represented by a $d$-dimensional cube of side length of $1$ which has a volume of $1^d = 1$. Then each polyhedron contains at most $\big\lfloor{V_j \over {1^d}}\big\rfloor + 1$ cubes (i.e points) from $\Gg_N$. Overall, $A$ contains at most
    \[
        \sum_{j=1}^D \big\lfloor{V_j}\big\rfloor + 1 \leq {1\over 2}(2a)^d + D
    \]
   points from $\Gg_N$. Since the number of points in $[-a,a]$ is exactly $(2a)^d$, we have
   \[
        \big|\Gg_N \cap A\big| \geq (2a)^d - {(2a)^d\over 2} - a\big({M^2 \over 2}(\sqrt{d}+2) + 2\big) > 0
   \]
\end{proof}

Throughout the rest of the section, for any $((h^{(j)},\tau^{(j)},n^{(j)})_{j=1}^M),\eta)$)-cluster configuration as in definition~\ref{def:multi_cluster}, $\tau^{(j)},n^{(j)},\eta,M$ and $d$ are assumed to be fixed parameters. Let $h := \max_j\{h^{(j)}\}_{j=1}^M$, we consider the asymptotic regime $Nh \ll 1$, where $N \rightarrow \infty$ and there exists a constant $0 < C \leq 1$ such that $h < {C \over N}$. In our main results, Theorems \ref{thm:multi-sv} and \ref{thm:entire-spec}, $C$ is not arbitrary and is upper bounded by some constant $C_7$. Finally, by a constant, we mean a quantity that is independent of $N$ and $h$.


\par From the above proposition, if we set
\[
    a := \left\lfloor {\sqrt[d-1]{M^2(\sqrt{d}/2+1)+1}\over 2} \right\rfloor + 1,
\]
then we can choose a ``good'' projection \( w \in \Gg_N \cap \qq_a\), where $w$ is bounded by a constant, which preserves the well-separation of the clusters after projecting them by $w$. We formalize this in the following proposition.
\begin{proposition}[Admissible Projection]\label{prop:proj}
     Let $\ci$ be a $((h^{(j)}, n^{(j)})_{j=1}^M, \eta)$ multi-cluster configuration. Then for ${\pi \over a} \leq \eta \leq 2\pi$ and $a:=\left\lfloor {\sqrt[d-1]{M^2(\sqrt{d}/2+1)+1}\over 2} \right\rfloor + 1$, there exists $w \in \Gg_N\setminus \{\boldsymbol{0}\}$ where $\|w\|$ is bounded only by a constant depending on $d$ and $M$, such that for all $\mathcal{C}_j, \mathcal{C}_k$ and $j\neq k$ we have
    {\begin{align*}
        \min_{j\neq k}\min_{x\in \mathcal{C}_j,\tilde{x}\in \mathcal{C}_k}\| \langle w,x\rangle - \langle w,\tilde{x}\rangle\|_{\mathbb{T}} \geq {1\over 4}{\sqrt{2}\pi \over (\sqrt{d}+1)M^2}.
    \end{align*}}
    We call such $w$ an \textbf{admissible projection}.
\end{proposition}
\begin{proof}
    Let $\mathcal{C}:=\{c_j\}_{j=1}^M$ such that $c_j:={1\over n^{(j)}}\sum_{x\in \mathcal{C}^{(j)}}x$ is the center of cluster $\mathcal{C}^{(j)}$. Moreover, let $\eta_c:= \min_{j,k}\|c_j - c_k\|_{\TT^d}$ and $\theta := \min_{j,k}\|c_j - c_k\|_{2}$. Since that $c_j \in \TT^d$ for all $j=1,\dots,M$ we have that $\theta \leq 2\pi$, $\eta_c = \min_{j,k}\|c_j - c_k\|_{\infty}$ and from the definition of $\eta$
    $$\eta \leq  \min_{j\neq k}\min_{x\in \mathcal{C}^{(j)}, \ y\in \mathcal{C}^{(k)}}\|x-y\|_{\TT^d} \leq \min_{j\neq k}\min_{x\in \mathcal{C}^{(j)}, \ y\in \mathcal{C}^{(k)}}\|x-y\|_{\infty}.$$ By norm inequalities we have 
    \[
        {\pi \over a} \leq \eta \leq \eta_c \leq \theta.
    \]
    Thus we have that $ {\pi \over a} \leq \theta \leq 2\pi$ and we can use Theorem \ref{thm:union-bound} on the set of centers $\mathcal{C}$. Combined with Proposition \ref{prop:proj_inG} and setting $a:=\left\lfloor {\sqrt[d-1]{M^2(\sqrt{d}+1)}\over 2} \right\rfloor + 1$ we get that there exists $w \in \Gg_N$ such that $\|w\|\leq a\sqrt{d}$ and for all $c_j,c_k \in \mathcal{C}, j\neq k$ we have
    {\begin{align*}
        \min_{j\neq k}\| \langle w,c_j\rangle - \langle w,c_k\rangle\|_{\mathbb{T}} = \theta_w, \ \theta_w \geq {1\over 2}{\sqrt{2}\pi \over (\sqrt{d}+1)M^2}.
    \end{align*}}
    For any $x \in \mathcal{C}_j$ there exists $r \in \mathbb{R}^d$ with $\|r\| \leq 1$ such that 
    $x = c_j + rh^{(j)}$, and similarly $\tilde{x} = c_k + \tilde{r}h^{(k)}$ for some $\tilde{r} \in 
    \mathbb{R}^d$ with $\|\tilde{r}\| \leq 1$. Let $x \in \mathcal{C}_j$ and $\tilde{x} \in 
    \mathcal{C}_k$ be the minimizing nodes, i.e.
    $$
        (x, \tilde{x}) \in \argmin_{\substack{j \neq k \\ y \in \mathcal{C}_j,\, \tilde{y} \in 
        \mathcal{C}_k}} \|\langle w, y\rangle - \langle w, \tilde{y}\rangle\|_{\mathbb{T}}.
    $$
    Then, using the inverse triangle inequality (which is valid since $\|\cdot\|_{\TT^d}$ is a metric)
\begin{align*}
    \min_{j\neq k}\min_{\substack{y\in \mathcal{C}_j \\ \tilde{y}\in \mathcal{C}_k}}
    \|\langle w,y\rangle - \langle w,\tilde{y}\rangle\|_{\mathbb{T}} 
    &= \|\langle w, x\rangle - \langle w, \tilde{x}\rangle\|_{\mathbb{T}} \\
    &= \|\langle w, c_j + rh^{(j)}\rangle - \langle w, c_k + \tilde{r}h^{(k)}\rangle\|_{\mathbb{T}} \\
    &= \|\langle w, c_j - c_k\rangle + \langle w, rh^{(j)} + \tilde{r}h^{(k)}\rangle\|_{\mathbb{T}} \\
    &\geq \big\|\|\langle w, c_j - c_k\rangle\|_{\mathbb{T}} - 
    \|\langle w, rh^{(j)} + \tilde{r}h^{(k)}\rangle\|_{\mathbb{T}}\big\|_{\mathbb{T}} \\
    &= \big\|\theta_w - \|\langle w, rh^{(j)} + \tilde{r}h^{(k)}\rangle\|_{\mathbb{T}}\big\|_{\mathbb{T}} \\
    &= \big\|\theta_w - |\langle w, rh^{(j)} + \tilde{r}h^{(k)}\rangle \mod (-\pi,\pi]|\big\|_{\mathbb{T}},
\end{align*}
    Now, because $|\langle w,rh^{(j)}+\tilde{r}h^{(k)}\rangle | \leq 2\|w\|\max \{h^{(j)},h^{(k)}\} < {\tilde{C}\over N} \ll {1\over 4}{\sqrt{2}\pi \over (\sqrt{d}+1)M^2}$ for some constant $\tilde{C}:=2\|w\|C$, we get the desired result. 
\end{proof}
Note that in the above Proposition, we considered the wrap-around distance from \prettyref{def:wrap_dist} between clusters and not the euclidean distance as in Theorems \ref{thm:slab} and \ref{thm:union-bound}.

\section{Vandermonde rows partitioning}\label{sec:van-row-par}
We now aim to prove the existence of \( K_1 \asymp N^{d-1} \) shifts such that the rows of the Vandermonde matrix can be partitioned into \(  K_1 \) matrices, each containing \( K_2 \asymp N \) rows. We begin with some definitions.

\begin{definition}\label{def:num-of-grid-p}
    For any polytope $T \subset \RR^d$, we define 
    \[
        \Gg(T) := \{p : p\in T \cap \ZZ^d\},
    \]
    as the grid points of the lattice $\ZZ^d$ in $T$. In addition, we denote by $|T|_{\ZZ^d}$ the number of grid points in $\Gg(T)$.
\end{definition}

\begin{definition}\label{def:oriented-cube}
    For any $w \in \mathbb{R}^d \setminus \{\boldsymbol{0}\}$, let
    $$v := \frac{N - \|w\|}{\|w\| \sqrt{d}} \cdot w$$
    be a scaled version of $w$, and let $\{e_1, e_2, \ldots, e_d\}$ be an orthonormal basis for 
    $\mathbb{R}^d$ where $e_1 = \frac{v}{\|v\|}$. Then the vector-oriented cube $\mathcal{Q}_w$ 
    is defined as:
    $$\qq_w := \left\{\sum_{i=1}^d t_i e_i : t_i \in \left[-\|v\|, 
    \|v\|\right] \right\}.$$
\end{definition}

The grid points in $\qq_w$ are going to be used to rearrange the rows of $U_N$ into 1D Vandermonde matrices. In the following lemma, we prove that the number of grid points in $\qq_w$ scales like $N^{d}$.
\begin{lemma}\label{cube-grid-p}
    For $d \leq {N - \|w\|\over 4}$ and $\|w\|\asymp 1$, we have that $|\qq_w|_{\ZZ^d} \asymp N^d$. In fact, $|\qq_w|_{\ZZ^d}$ is at least $\big\lfloor|\qq_{v}|\big\rfloor$ where 
    \begin{align*}
        &\qq_{v} := \left\{\sum_{i=1}^d t_i e_i : t_i \in \left[-\|\tilde{v}\|, \|\tilde{v}\|\right] \text{ for all } i\right\}, \quad \tilde{v} = {\|v\|-2\sqrt{d} \over \|v\|}v.
    \end{align*}
\end{lemma}
\begin{proof}
    For $\qq_w$, the number of grid points is lower bounded by the number of unit cubes that are fully contained in $\qq_w$ and are aligned with the original coordinate axes (not $e_1,\dots,e_d$). So we first show that every unit cube that intersects $\qq_{v}$ is fully contained in $\qq_w$. Let $y$ be a point on the grid such that its corresponding unit cube $Q_y$ is not fully contained in $\qq_{v}$, just intersects it. Let $B_y$ be the smallest $d$-dimensional ball that contains $Q_y$. The diameter of $B_y$ is $\sqrt{d}$. To see that $\mathcal{Q}_v$ and $\mathcal{Q}_w$ are at distance at least $\sqrt{d}$ apart, note that by definition of $\tilde{v}$, the sides of $\mathcal{Q}_v$ are smaller than those of $\mathcal{Q}_w$ by $2\sqrt{d}$, and both cubes share the same center and orientation. Hence every point in $\mathcal{Q}_v$ is at least $\sqrt{d}$ from the complement of $\mathcal{Q}_w$. Since $Q_y$ intersects $\mathcal{Q}_{v}$, there exists a point $p \in Q_y \cap \mathcal{Q}_{v}$. For any other point $q \in Q_y$, we have $\|p - q\| \leq \sqrt{d}$, since $\sqrt{d}$ is the diameter of $B_y$. Since every point in $\mathcal{Q}_v$ is at least $\sqrt{d}$ from the complement of $\mathcal{Q}_w$, any point $q \in Q_y$ satisfies:
    $$\|q - p\| \leq \sqrt{d} \implies q \in \mathcal{Q}_w.$$
    Since $q \in Q_y$ was arbitrary, we conclude that $Q_y \subseteq \mathcal{Q}_w$.
    \par From the above, we conclude that the number of unit cubes that intersect (or fully contained in) $\qq_{v}$ is lower bounded by $\big\lfloor|\qq_{v}|\big\rfloor$. Thus, on the one hand, 
    \[
        |\qq_{v}| = \|\tilde{v}\|^d = \big\| {\|v\|-2\sqrt{d} \over \|v\|}v\big\|^d = \big({N - \|w\|\over \sqrt{d}} - 2\sqrt{d}\big)^d \geq \big({N \over 2\sqrt{d}}\big)^d
    \]
    for $d \leq {N - \|w\|\over 4}$, which means that $|\qq_w|_{\ZZ^d} \geq \big\lfloor\big({N \over 2\sqrt{d}}\big)^d\big\rfloor$. On the other hand, $$|\qq_w|_{\ZZ^d} \leq |\qq_w| = \|v\|^d = \big({N - \|w\|\over \sqrt{d}}\big)^d.$$ Thus, we get the desired result. For an illustration of this proof, see figure \ref{fig:grid-lemma}.
\end{proof}
\begin{figure}
    \centering
    \includegraphics[width=0.5\linewidth]{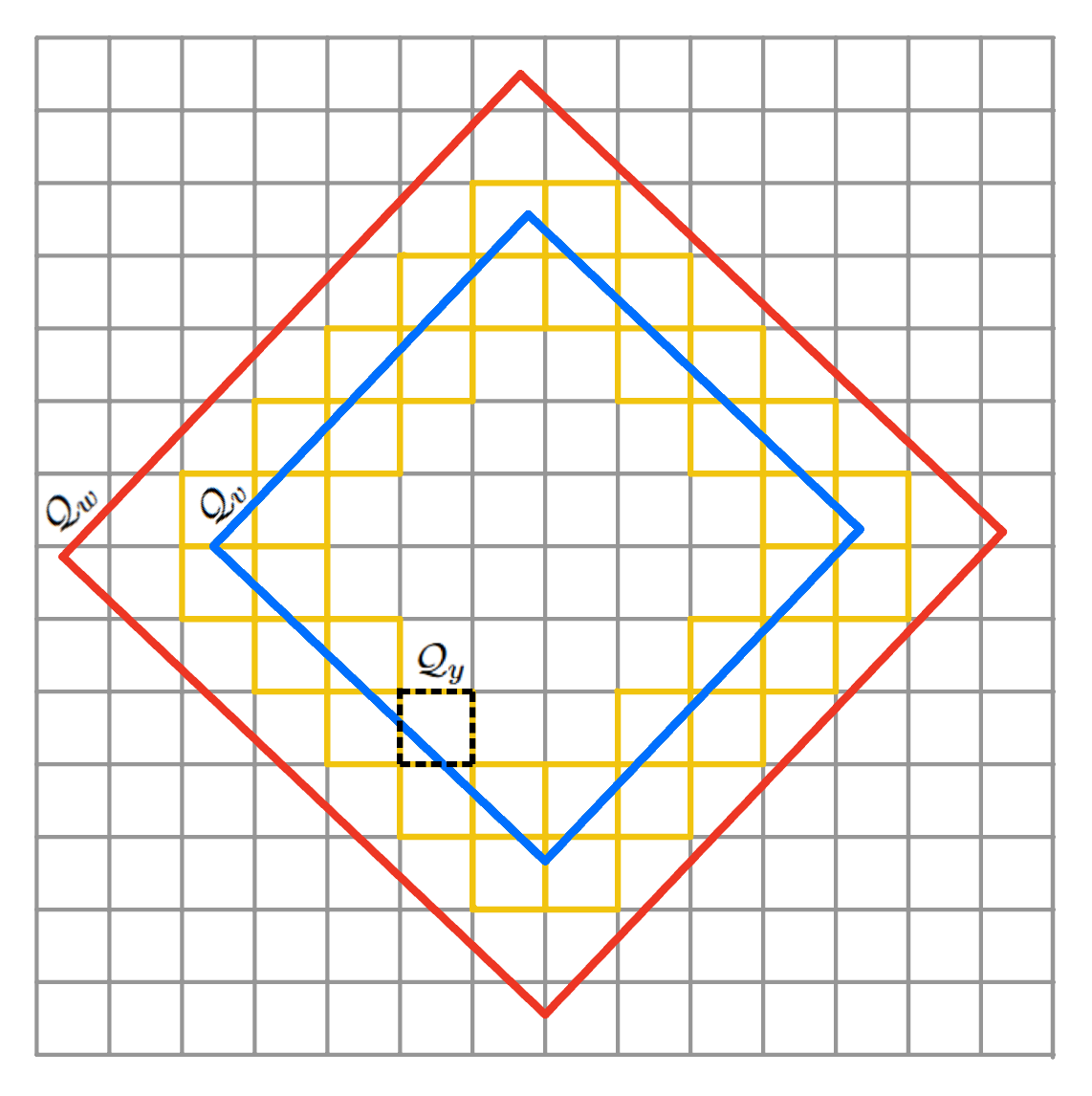}
    \caption{Illustration of the proof of Lemma \ref{cube-grid-p} for $d=2$. Red cube is $\qq_w$, Blue cube is $\qq_{v}$ and each yellow unit cube corresponds to $Q_y$. Each unit cube represents a grid point in $\Gg_N$ at its center.}
    \label{fig:grid-lemma}
\end{figure}

\begin{definition}[Lattice ray]\label{def:l-r}
    For any $w, s \in \RR^d$, we define the Lattice-ray $L_{w,s} := \{kw + s : k\in \mathbb{Z}\}$.
\end{definition}

    \begin{definition}[slice]
        For a unit vector $u\in\RR^d$ and $t,w\in \RR$, we define $S_u(t,w)$ to be  the $d$-dimensional slice  
        \[
            S_u(t,w) := \{x\in \RR^d : t\leq \langle x,u \rangle < t+w\}.
        \]
        Equivalently, $S_u(t+w)$ is a slice of width $w$ defined by two parallel hyperplanes orthogonal to $u$.
    \end{definition}

Now we prove the main result of this section. We show that $\qq_w$ contains $K_1 \asymp N^{d-1}$ lattice-rays, where each lattice-ray contains $K_2 \asymp N$ grid points and represents a 1D Vandermonde matrix.
\begin{theorem}\label{thm:van-part}
    For any projection $w \in \Gg_N\setminus \{\boldsymbol{0}\}$ such that $\|w\| \asymp 1$, there exists $K_1 \asymp N^{d-1}$ shifts $\{s_i\}_{i=1}^{K_1} \subset \Gg_N$ such that the rows of $U_N$ can be partitioned into: 
    \begin{enumerate}
        \item $K_1$ univariate Vandermonde matrices 
                \begin{equation}\label{eq:projected-vand-1d}
                    V_{s_i} = \big[\exp{(\imath \langle \ell, x\rangle)} \big]_{\ell \in I_{w,s_i}}^{x\in \ci},\qquad I_{w,s_i} = L_{w,s_i}\cap \qq_w
                \end{equation}
              of dimension $K_2^{(i)} \times n$ where $K_2^{(i)} = |I_{w,s_i}| \asymp N$ for $i=1,\dots, K_1$. 
        \item A residual matrix $$R_w =  \big[\exp{(\imath \langle t, x\rangle)}\big]_{t \in \Gg_N \setminus 
        \Gg(\qq_w)}^{x\in \ci}, \qquad \Gg(\qq_w) = \bigcup_{i=1}^{K_1} I_{w,s_i}$$ of dimension $({(2N+1)}^d - \sum_{i=1}^{K_1}K_2^{(i)})\times n$.
    \end{enumerate}
   
\end{theorem}

\begin{proof}
    \par Let $\qq_N:=[-N,N]^d$. We show that for any $w$, $\qq_w$ is fully contained in $\qq_N$. Since $\qq_w$ is centered at the origin, it is enough to show that the maximal diameter (of the circumscribed ball) of the cube $\qq_w$ is equal or smaller than the side length of $\qq_N$. The maximal diameter of $\qq_w$ is of length $\sqrt{d}\cdot \|v\| = N-\|w\|$ which is smaller than the side length of $\qq_N$, equal to $N$. 
    \par We now prove that $\qq_w$ contains $K_1 \asymp N^{d-1}$ Lattice-rays $\{L_{w,s_i}\}_{i=1}^{K_1}$ each containing $K_2 \asymp N$ grid points.
    Using Lemma \ref{cube-grid-p}, we have $|\qq_w|_{\ZZ^d} \asymp N^{d}$. For $u:={w \over \|w\|}$ and any $t\in \RR$, let $A_w(t) := S_u(t,\|w\|)\cap \qq_w$. Since the grid points are uniformly distributed in $\qq_w$ and using the pigeonhole principle, there exists $t$ such that for any $A_w \subset \qq_w$ we have  
    $$\bigg\lfloor\frac{|\qq_w|_{\ZZ^d}}{\|v\|}\|w\|\bigg\rfloor \leq |A_w(t)|_{\ZZ^d} \leq \bigg\lceil\frac{|\qq_w|_{\ZZ^d}}{\|v\|}\|w\|\bigg\rceil.$$  
    Without loss of generality, we choose $t=0$ and let $A_w:=A_w(0)$. Thus, we have that  $K_1 := |A_w|_{\ZZ^d} \asymp N^{d-1}$. 
    \par Let $H_0, H_w$ denote the $d-1$ dimensional hyperplanes that are orthogonal to $w$ that pass through the grid point $0$ (i.e. the origin) and $w$ respectively. Note that $H_0 \subset A_w$ and $H_w \not\subset A_w$ and that each (grid) point in $\Gg(H_0)$ belongs to a Lattice ray of a point in $\Gg(H_w)$, because the distance between $H_0$ and $H_w$ is $\|w\|$ and $H_0$ and $H_w$ are parallel.
    \par Now we show that each grid point in $A_w$ belongs to a different Lattice-ray $L_{w,s_i}$ (i.e. $\{s_i\}_{i=1}^{K_1} = \Gg(A_w)$) and all Lattice-rays $\{L_{w,s_i}\}_{i=1}^{K_1}$ are disjoint. Let $p_1,p_2 \in A_w$. Then we can write
    \[
        p_i = m_{i,1}w + s_i, \ s_i = \sum_{j=2}^d m_{i,j}e_j, 
    \]
    where $m_{i,j} \in \ZZ$ for $i=1,2$ and $j=1,\dots,d$. If $m_{1,1} = m_{2,1}$ then $p_1$ and $p_2$ belong to different lattice-rays since $s_1 \neq s_2$ (i.e. $p_i \in L_{w,s_i}$). Otherwise, if $m_{1,1} \neq m_{2,1}$, we assume that $p_1,p_2\in L_{w,s}$. Then, we have 
    \begin{align*}
        &p_1 = m_{1,1}w + s, \ p_2 = m_{2,1}w + s\\
        &\Rightarrow p_1 - p_2 = (m_{1,1} - m_{2,1})w,
    \end{align*}
    where $|m_{1,1} - m_{2,1}| \geq 1$. A contradiction, since $p_1,p_2\in A_w$ and $H_w \not\subset A_w$.

    To conclude, we get that $\qq_w$ has at least $|A_w|_{\ZZ^d}$ Lattice-rays and each $s_i\in \Gg(A_w), i=1,\dots,K_1$ is a Lattice-ray candidate. 
    \par Let $K_2^{(i)} = |I_{w,s_i}| = |L_{w,s_i} \cap \qq_w|$, we have that $K_2^{(i)} \asymp N$ because there are at least $\bigg\lfloor{\|v\| \over \|w\|}\bigg\rfloor = \bigg\lfloor{N-\|w\| \over \sqrt{d}\|w\|}\bigg\rfloor$ grid points in $L_{w,s_i} \cap \qq_w$ and at most $\bigg\lfloor{\|v\| \over \|w\|}\bigg\rfloor + 1 = \bigg\lfloor{N - \|w\| \over \sqrt{d}\|w\|}\bigg\rfloor + 1$. Note that since $w\in\Gg_N\setminus\{0\}$, in particular, $\|w\|\geqslant 1$. Overall, we get the result of part 1 in the theorem and the second part immediately follows.
    For an illustration of the proof see figure \ref{fig:lattice-rays}.
\end{proof}
\begin{figure}[t]
\centering
\includegraphics[width=0.70\textwidth]{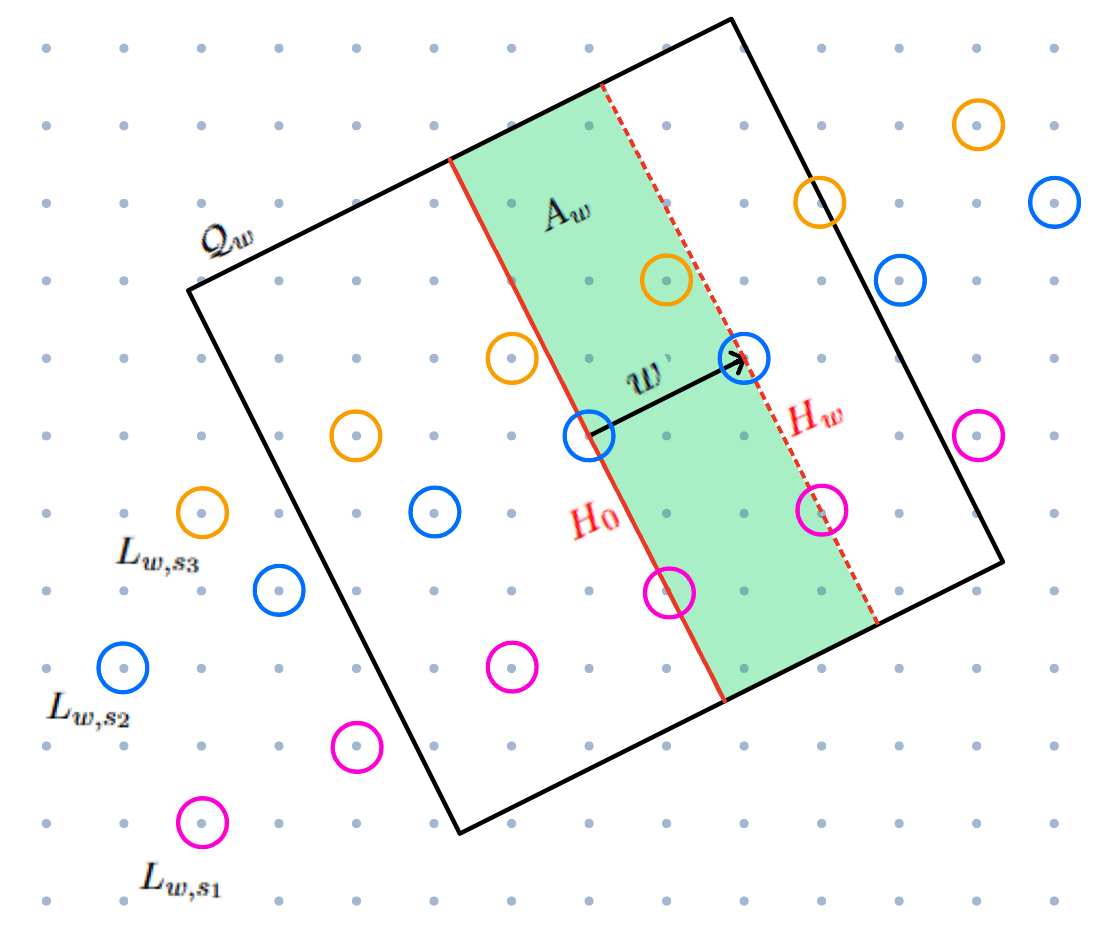}
\caption{Illustration of the proof of \prettyref{thm:van-part} for $d=2$.}
\label{fig:lattice-rays}
\end{figure}

\begin{remark}
    The proofs of Lemma~\ref{cube-grid-p} and Theorem~\ref{thm:van-part} remain valid for $v := \frac{N}{\|w\| \sqrt{d}} \cdot w$; our choice in Definition~\ref{def:oriented-cube} is tailored to Lemma~\ref{lem:trig_cancel}.
\end{remark}

\section{Cluster subspace orthogonality}\label{sec:cl-sub-orth}
In this section we prove that if two clusters $\mathcal{C}_j,\mathcal{C}_k\subseteq\ci$ satisfy suitable separation conditions, then the angle between the subspaces spanned by $U_N(\mathcal{C}_j)$ and $U_N(\mathcal{C}_k)$ is bounded above by $${\pi\over 2}-{\alpha\over N}-\beta Nh - {\gamma\over \sqrt{N}}$$ for some constants $\alpha,\beta,\gamma$. In this case we say these subspaces are near-orthogonal.
Recall some definitions from \cite{batenkov2021}, 
\begin{definition}\label{def:col-span-1d}
    For a set $Y = \{y_1,...,y_m\} \subset \RR$ and finite $Z \subset \mathbb{Z}$, we define 
    \begin{equation}
        L(Y,Z) = \operatorname{span}\{v_1,...,v_m\}, \quad v_k := \big(e^{izy_k}\big)_{z\in Z}.
    \end{equation}
\end{definition}
\begin{definition}
    Denote by $L := L(Y,Z)$ and $\Tilde{L} := L(\Tilde{Y},Z)$. Then,
    \begin{equation}
        \measuredangle_{\min}(L,\Tilde{L}) = \min_{\small{v \in L\setminus\{0\}, \tilde{v} \in \Tilde{L}\setminus\{0\} }}\arccos\biggl({{|\langle v, \tilde{v}\rangle|}\over {\|v\| \cdot \|\tilde{v}\|}}\biggr)
    \end{equation}
    
\end{definition}

\begin{definition}[Cluster subspace]
    Let $\ci = \{x_1,...,x_n\} \subset \TT^d$ and let $\uu_1,...,\uu_n$ denote the columns of the multivariate Vandermonde matrix $U_N(\ci)$. We denote by $\Ll(\ci)$ the subspace spanned by $\uu_1,...,\uu_n$, 
    \begin{equation}
        \Ll(\ci) := \Ll(\ci,\Gg_N) = \operatorname{span}\{\uu_1,\dots,\uu_n\} \subset \mathbb{C}^{(2N+1)^d}.
    \end{equation}
\end{definition}

\begin{definition}
    Denote by $\Ll := \Ll(\ci,\Gg_N)$ and $\Tilde{\Ll} := \Ll(\Tilde{\ci},\Gg_N)$. Then,
    \begin{equation}
        \measuredangle_{\min}(\Ll,\Tilde{\Ll}) = \min_{\small{\uu \in \Ll\setminus\{0\}, \tilde{\uu} \in \Tilde{\Ll}\setminus\{0\} }}\arccos\biggl({{|\langle \uu, \tilde{\uu}\rangle|}\over {\|\uu\| \cdot \|\tilde{\uu}\|}}\biggr).
    \end{equation}
    
\end{definition}
In the following corollary, we restate Theorem \ref{thm:van-part} after truncating each univariate Vandermone matrix $V_{s_i}$ in \eqref{eq:projected-vand-1d} to the length of the shortest lattice ray; the discarded rows are added to the residual matrix $R_w$.
\begin{corollary}\label{cor:van-vec}
    Let
    \begin{align*}
    K_2^* &= \min_{i}\{K_2^{(i)}\}_{i=1}^{K_1},\\
    K_2 &= \begin{cases}
    K_2^* & \text{if } K_2^* \text{ is odd}, \\
    K_2^* - 1 & \text{if } K_2^* \text{ is even}.
    \end{cases}.
    \end{align*}
    Moreover, put $Z_{K_2}:=\{-{K_2-1 \over 2},\dots,{K_2-1 \over 2}\}$ and denote by $\mathcal{I}_{w,s} := \{kw + s: k \in Z_{K_2}\}$. Then, with a slight abuse of notation,
    following Theorem \ref{thm:van-part}, there exists a permutation matrix $P\in \RR^{(2N+1)^d \times (2N+1)^d}$ such that each column $\uu_r$ of $U_N$ can be written as follows:
    \begin{align*}
         \uu_r &= P\cdot \begin{pmatrix}
                                V_{s_1}(x_r) \\
                                \vdots \\
                                V_{s_{K_1}}(x_r) \\
                                R_w(x_r)
                        \end{pmatrix}, \\ 
        V_{s_i}(x_r) &= \big[\exp{(\imath \langle \ell, x_r\rangle)}\big]_{\ell \in \mathcal{I}_{w,s_i}} \\&= \big[e^{\imath \langle s_i, x_r\rangle} e^{\imath \langle kw, x_r\rangle}\big]_{k\in Z_{K_2}} \ \text{(univariate Vandermone column)}, \\
        \\ R_w(x_r) &=  \big[\exp{(\imath \langle t, x_r\rangle)}\big]_{t \in \Gg_N \setminus 
        \big(\bigcup_{i=1}^{K_1} \mathcal{I}_{w,s_i}\big)}.
    \end{align*}
  
\end{corollary}
Now we state the main result of this section.
\begin{theorem}[Cluster subspace orthogonality]\label{thm:orth}
    Let $\ci$ and $\tilde{\ci}$ form an $(h_1,n_1)$ and $(h_2,n_2)$ clusters respectively where
    \[
        \|x-y\|_{\td} \geq \eta, \\ \forall x \in \ci, \tilde{x}\in \tilde{\ci}.
    \]
    Put $a=\left\lfloor {\sqrt[d-1]{4(\sqrt{d}/2+1)+1}\over 2} \right\rfloor + 1$. Then for ${\pi\over a} \leq \eta \leq 2\pi$, there exist an admissible projection $w\in \Gg_N\setminus \{\boldsymbol{0}\}$ with $\|w\|$ bounded by a constant depending only on $d$ and positive constants $C_2, C_3,C_4$ depending only on $s_1, s_2$ and $C_1, C_5$ depending also on $d$, such that for all $N$ with $C_1 < N \leq {C_2 \over h}$ we have
    \[
        \measuredangle_{\min}(\Ll(\ci,\Gg_N),\Ll(\tilde{\ci},\Gg_N)) \geq {\pi \over 2} - {C_4 \over N\theta_w} - C_3Nh - {C_5 \over N^{1 \over 2}},
    \]
    where $h = \max\{h_1,h_2\}$ and $\theta_w :=  \min_{x\in \ci, \tilde{x}\in \tilde{\ci}}\|\langle w,x\rangle - \langle w,\tilde{x}\rangle\|_\TT$.
\end{theorem}

\begin{proof}
Using Proposition \ref{prop:proj}, for two clusters $\ci$ and $\tilde{\ci}$ (i.e. $M=2$) and ${\pi\over a} \leq \eta \leq 2\pi$, we have that there exists an admissible projection $\omega \in \Gg_N\setminus \{\boldsymbol{0}\}$ with $\|w\|$ depending only on $d$. 
\par Now, let $Y = \{y_1,...,y_{n_1}\}, \ y_j = \langle w,x_j\rangle$, $\tilde{Y} = \{\tilde{y}_1,...,\tilde{y}_{n_2}\}, \ \tilde{y}_j = \langle w,\tilde{x}_j\rangle$.
From corollary \ref{cor:van-vec}, we have 
\[
    \uu_j = P\cdot\begin{pmatrix}
            q_1v_j \\  \vdots \\ q_{K_1}v_j \\ R_{w}(x_j)
        \end{pmatrix}, \ \tilde{\uu}_j = P\cdot\begin{pmatrix}
             \tilde{q}_1 \tilde{v}_j \\  \vdots \\ \tilde{q}_{K_1} \tilde{v}_j \\ R_{w}(\tilde{x}_j)
        \end{pmatrix}, \ v_j = [e^{\imath ky_j}\big]_{k\in Z_{K_2}}, \ \tilde{v}_j = [e^{\imath k\tilde{y}_j}\big]_{k\in Z_{K_2}}
\]
where $q_i := q_i^{(j)} = e^{\imath \langle s_i,x_j\rangle}$ and $\tilde{q}_i, 
\;\tilde{q}_{i}^{(j)}$ defined similarly. For any $\uu \in \Ll\setminus\{0\}, \tilde{\uu} \in \tilde{\Ll}\setminus\{0\}$ we have
\begin{align*}
    {{|\langle \uu, \tilde{\uu}\rangle|}\over {\|\uu\| \cdot \|\tilde{\uu}\|}} &= {{|\langle \sum_{j=1}^{n_1}a_j\uu_j, \sum_{t=1}^{n_2}\tilde{a}_t\tilde{\uu}_t\rangle|}\over {\|\sum_{j=1}^{n_1}a_j\uu_j\| \cdot \|\sum_{t=1}^{n_2}\tilde{a}_t\tilde{\uu}_t\|}} \\
   &=  {{| \sum_{j=1}^{n_1}a_j\sum_{t=1}^{n_2}\tilde{a}_t^*\langle \uu_j, \tilde{\uu}_t\rangle|}\over {\sqrt{\sum_{j=1}^{n_1}a_j\sum_{t=1}^{n_1}a_t^*\langle \uu_j, {\uu}_t\rangle} \cdot \sqrt{\sum_{j=1}^{n_2}\tilde{a}_j\sum_{t=1}^{n_2}\tilde{a}_t^*\langle \tilde{\uu}_j, \tilde{\uu}_t\rangle}}} \\
    &\leq {{| \sum_{j=1}^{n_1}a_j\sum_{t=1}^{n_2}\tilde{a}_t^*\big(\sum_{k=1}^{K_1}q_k\tilde{q}_k^*\langle v_j, \tilde{v}_t\rangle\big) + \langle R_w(x_j), R_w(\tilde{x}_t)\rangle)|}\over {\sqrt{\sum_{j=1}^{n_1}a_j\sum_{t=1}^{n_1}a_t^*\big(\sum_{k=1}^{K_1}q_k{q}_k^*\langle v_j, v_t\rangle\big)} \cdot \sqrt{\sum_{j=1}^{n_2}\tilde{a}_j\sum_{t=1}^{n_2}\tilde{a}_t^*\big(\sum_{k=1}^{K_1}\tilde{q}_k\tilde{q}_k^*\langle \tilde{v}_j, \tilde{v}_t\rangle\big)}}} 
\end{align*}
Now denote by $v = \sum_{j=1}^{n_1}a_jv_j$, $\tilde{v} = \sum_{t=1}^{n_2}\tilde{a}_t\tilde{v}_t$ and let $a := (a_1,...,a_{n_1})^T, \  \tilde{a} := (\tilde{a}_1,...,\tilde{a}_{n_2})^T$, we have 
\begin{align*}
    {{|\langle \uu, \tilde{\uu}\rangle|}\over {\|\uu\| \cdot \|\tilde{\uu}\|}} &\leq \frac{K_1|\langle v, \tilde{v}\rangle| + | \sum_{j=1}^{n_1}a_j\sum_{t=1}^{n_2}\tilde{a}_t^*\langle R_w(x_j), R_w(\tilde{x}_t)\rangle|}{K_1\|v\|\|\tilde{v}\|} \\
    &= \frac{|\langle v, \tilde{v}\rangle|}{\|v\|\|\tilde{v}\|} + \frac{| \sum_{j=1}^{n_1}a_j\sum_{t=1}^{n_2}\tilde{a}_t^*\langle R_w(x_j), R_w(\tilde{x}_t)\rangle|}{K_1\|v\|\|\tilde{v}\|} \\
   & \leq \frac{|\langle v, \tilde{v}\rangle|}{\|v\|\|\tilde{v}\|} + \frac{\sqrt{n_1n_2}\|a\|\|\tilde{a}\|\max_{j,t}|\langle R_w(x_j), R_w(\tilde{x}_t)\rangle|}{K_1\sqrt{\sum_{j=1}^{n_1}a_j\sum_{t=1}^{n_1}a_t^*\langle v_j, v_t\rangle} \cdot \sqrt{\sum_{j=1}^{n_2}\tilde{a}_j\sum_{t=1}^{n_2}\tilde{a}_t^*\langle \tilde{v}_j, \tilde{v}_t\rangle}},
\end{align*}
where the last inequality was obtained using Cauchy-Schwartz. Let $\gamma_{j,t} := \langle v_j, v_t\rangle, \ \tilde{\gamma}_{j,t} := \langle \tilde{v}_j, \tilde{v}_t\rangle$, using Lemma \ref{lem:cluster-inner-prod}, we have $\min_{j,t}\gamma_{j,t} \geq c_3 \cdot N$ and $\min_{j,t}\tilde{\gamma}_{j,t} \geq \tilde{c}_3 \cdot N$, for some constants $c_3,\tilde{c}_3 > 0$. 
Finally we get
\begin{align*}
    {{|\langle \uu, \tilde{\uu}\rangle|}\over {\|\uu\| \cdot \|\tilde{\uu}\|}} 
    &\leq \frac{|\langle v, \tilde{v}\rangle|}{\|v\|\|\tilde{v}\|} + \frac{\sqrt{n_1n_2}\|a\|\|\tilde{a}\|\max_{j,t}|\langle R_w(x_j), R_w(\tilde{x}_t)\rangle|}{K_1\|a\|\|\tilde{a}\|\sqrt{\min_{j,t}\gamma_{j,t}\min_{j,t}\tilde{\gamma}_{j,t}}} \\
    &\leq \frac{|\langle v, \tilde{v}\rangle|}{\|v\|\|\tilde{v}\|} + \frac{\sqrt{n_1n_2}\max_{j,t}|\langle R_w(x_j), R_w(\tilde{x}_t)\rangle|}{K_1\sqrt{c_3\tilde{c}_3}N}.
\end{align*}
Let $j, t$ be such that
$$|\langle R_w(x_{j}), R_w(\tilde{x}_{t})\rangle| = \max_{k,m} |\langle R_w(x_k), R_w(\tilde{x}_m)\rangle|,$$
and let $\phi = x_j - \tilde{x}_t$. Using Lemma \ref{lem:trig_cancel}, for some constant $c_4$ we have 
\[
        \bigg|\langle R_w(x_j), R_w(\tilde{x}_t)\rangle\bigg| = \bigg| \sum_{s\in \overline{\Gg(\qq_w)}}e^{i\langle s,\phi \rangle}\bigg| \leq c_4\cdot N^{d-1}{1\over |1-e^{i\langle \omega,\phi \rangle}|}.
    \]
We now show that $|1-e^{i\langle \omega,\phi \rangle}| \geq c_5$. Recall Lemma 5.6 in \cite{batenkov2021a}, for $|\alpha|=|\beta|=1$ we have 
\[
    {2\over \pi}\text{Arg}\big({\alpha\over\beta}\big) \leq |\alpha - \beta | \leq \text{Arg}\big({\alpha\over\beta}\big).
\]
Since $w$ is an admissible projection we have that $\|\langle w,x_j-\tilde{x}_t\rangle\|_{\TT} = \|\langle w,\phi\rangle\|_{\TT} \geq {\sqrt{2}\pi \over 16(\sqrt{d}+1)}$ (by proposition \ref{prop:proj} with $M=2$), thus:
\[
    {2\over \pi}{\sqrt{2}\pi \over 16(\sqrt{d}+1)} \leq {2\over \pi}\|\langle w,\phi \rangle\|_{\TT} = {2\over \pi}\text{Arg}({e^{i\langle w,\phi \rangle}\over 1})  \leq \big|{{1-e^{i\langle w,\phi \rangle}}}\big| \leq \text{Arg}({e^{i\langle w,\phi \rangle}\over 1}) = \|\langle w,\phi \rangle\|_{\TT} \leq 2\pi,
\]
and consequently, $\big|{1 \over {1-e^{i\langle w,\phi \rangle}}}\big| \leq {1\over c_5}$ for some constant $c_5$ that doesn't depend on $N$ or $h$. 
To conclude, we have 
    \[
        \bigg|\langle R_w(x_j), R_w(\tilde{x}_t)\rangle\bigg| \leq  {c_4 \over c_5} \cdot N^{d-1}.
    \]
For $K_1 \geq c_1N^{d-1}$, put $c_6 := {c_4 \over c_5} {\sqrt{n_1n_2}\over c_1\sqrt{c_3\tilde{c}_3}}$, then
 \begin{align*}
     \min_{\small{\uu \in \Ll\setminus\{0\}, \tilde{\uu} \in \Tilde{\Ll}\setminus\{0\} }}\arccos\biggl({{|\langle \uu, \tilde{\uu}\rangle|}\over {\|\uu\| \cdot \|\tilde{\uu}\|}}\biggr) 
     \geq \min_{\small{v \in L\setminus\{0\}, \tilde{v} \in \Tilde{L}\setminus\{0\} }}\arccos\biggl({{|\langle v, \tilde{v}\rangle|}\over {\|v\| \cdot \|\tilde{v}\|}} + \frac{c_6}{N}\biggr),
 \end{align*}
where $L = L(Y,Z_{K_2})$, $\tilde{L} = L(\tilde{Y},Z_{K_2})$ are given by \prettyref{def:col-span-1d}.

Let $\rho = \frac{|\langle v, \tilde{v}\rangle|}{\|v\|\|\tilde{v}\|}$ and $\nu =  {c_6 \over N}$. Using Taylor expansion of $\arccos$,  we have for $\xi \in (\rho, \rho+\nu)$
    \begin{align*}
        \arccos(\rho + \nu) = \arccos(\rho) - \arccos'(\xi)\nu &= \arccos(\rho) - {\nu \over \sqrt{1-\xi^2}} \geq \arccos(\rho) - {\nu \over \sqrt{1-{(\rho+\nu)}^2}},
    \end{align*}
    and assuming that $\rho < 1-2\nu$, for $C_5 := \sqrt{c_6 \over 2}$, we get
    \begin{align*}
        \arccos(\rho + \nu) &\geq \arccos(\rho) - {\nu \over \sqrt{1-{(1-2\nu+\nu)}^2}} = \arccos(\rho) - {\nu \over \sqrt{2\nu - \nu^2}} \geq \arccos(\rho) - {\sqrt{\nu \over 2}} \\
        & = \arccos\big(\frac{|\langle v, \tilde{v}\rangle|}{\|v\|\|\tilde{v}\|}\big) - {\sqrt{\nu \over 2}} = \arccos\big(\frac{|\langle v, \tilde{v}\rangle|}{\|v\|\|\tilde{v}\|}\big) - {C_5 \over N^{1 \over 2}}.
    \end{align*}
    Now since $v, \tilde{v}$ are columns of the univariate Vandermonde matrix for $((h^{(i)},n_i)_{i=1,2},\theta_w)$-cluster configuration, by Theorem 2.1 in \cite{batenkov2021} there exist positive constants $\tilde{C}_1, C_2, C_3$ and $C_4$ depending only on $s_1, s_2$, such that for all $N$ with $\tilde{C}_1 \leq N \leq {C_2 \over h}$ we have
    \begin{equation}\label{orth-1d}
         \min_{\scriptsize{v \in L\setminus\{0\}, \tilde{v} \in \Tilde{L}\setminus\{0\} }}\arccos\big({{|\langle v, \tilde{v}\rangle|}\over {\|v\| \cdot \|\tilde{v}\|}}\big) \geq {\pi \over 2} - {C_4 \over N\theta_w} - C_3Nh.
    \end{equation}
  
    To finish the proof, we show that for $N > {C_4(\theta_w)^{-1} + 4C_5^2 \over 1+C_3C_2}$, the assumption $\rho < 1-2\nu$ holds.
    Using \eqref{orth-1d}, we have
    \begin{align*}
        \rho \leq \cos{\big({\pi \over 2} - {C_4 \over N\theta_w} - C_3Nh\big)},
    \end{align*}
    using Taylor series of $\cos(x)$ around $\pi \over 2$, for some $\xi \in ({\pi \over 2} - {C_4 \over N\theta_w} - C_3Nh,{\pi \over 2})$ we have
    \begin{align*}
        \rho &= \cos({\pi \over 2}) -\sin({\pi \over 2})(-{C_4\over N\theta_w} - C_3Nh) - {\cos(\xi)\over 2}({C_4\over N\theta_w} + C_3Nh)^2 \\
             &\leq {C_4\over N\theta_w} + C_3Nh \leq {C_4\over N\theta_w} + C_3C_2 
    \end{align*}
   and we find $N$ such that 
   \[
        \rho \leq {C_4\over N\theta_w} + C_3C_2 < 1 - 2{c_6 \over N}
   \]
    holds and put $C_1 := \max\{\tilde{C}_1, {C_4(\theta_w)^{-1} + 4C_5^2 \over 1+C_3C_2}\}$.

\end{proof}

\begin{remark}
    For $d=1$, the bound \eqref{orth-1d} established in Theorem 2.1 in \cite{batenkov2021}  is tighter than the bound established in Theorem \ref{thm:orth}, since the former does not use any redundant projections.
\end{remark}
\begin{remark}\label{rem:inner-cluster-proj}
    As seen in the proof of Theorem~\ref{thm:orth}, the internal geometry of each cluster under an admissible projection is irrelevant; even if nodes within a cluster collide, near-orthogonality is determined solely by the separation between the clusters and the total cluster extent.
\end{remark}
\subsection{Auxiliary lemmas}
We give the proof of lemmas used in Theorem \ref{thm:orth}.
    \begin{lemma}\label{lem:cluster-inner-prod}
        Let $\ci$ form a $(h,n)$ cluster, then for any $x_j,x_t \in \ci$, there exists a constant $c_3$ (independent of $N,h$) such that
        \begin{equation}
            \gamma_{j,t} :=\langle v_j, v_t \rangle \geq c_3 \cdot N,
        \end{equation}
        where $v_j = (e^{kx_j})_{k\in Z_{K_2}}, v_t = (e^{kx_t})_{k\in Z_{K_2}}$ (two columns of a univariate Vandermonde matrix) and $K_2 \asymp N$ is odd.
    \end{lemma}
    \begin{proof}
         Let $h_{j,t} := \|x_j - x_t\|_{\TT^d}$ and $N_2 = {K_2 - 1 \over 2}$. We have
        \begin{align*}
            \gamma_{j,t}=\langle v_j, v_t \rangle = \sum_{k=-N_2}^{N_2}e^{\imath k(x_j - x_t)} = \frac{\sin((N_2+{1\over 2})h_{j,t})}{\sin{({h_{j,t} \over 2})}}.
        \end{align*}
        Since $x_j,x_t$ are in the same cluster we have that $h_{j,t} \leq h$ and $Nh \leq 1$ thus we have $$|(N_2+{1\over 2})h_{j,t}| \leq |Nh + {h\over 2}| \leq |Nh + {1\over 2N}| \leq {\pi\over 2}.$$
        Moreover, for $|\theta| \leq {\pi \over 2}$ we have that ${2\over \pi}|\theta| \leq \sin |\theta | \leq |\theta|$. Thus we get
        \[
            {4\over \pi}(N_2 + {1\over 2}) \leq \frac{\sin((N_2+{1\over 2})h_{j,t})}{\sin{({h_{j,t} \over 2})}} \leq \pi(N_2 + {1\over 2}),
        \]
        and in particular the desired result.
    \end{proof}

\begin{lemma}\label{lem:trig_cancel}[Trigonometric Cancellation]
    For $\omega \in \Gg_N$ such that $\|w\| \leq c$ for some constant $c:=c(d,M)$ and any $t \neq 0$, there exists a constant $c_4$ such that
    \[
        \bigg| \sum_{s\in {\Gg_N\setminus\Gg(\qq_w)}}e^{i\langle s,t\rangle}\bigg| \leq c_4N^{d-1}{1\over |1-e^{i\langle \omega,t\rangle}|}
    \]
\end{lemma}
where $\Gg(\qq_w) := \{p : p\in \qq_w \cap \ZZ^d\}$ (as in \prettyref{def:num-of-grid-p}) and $\qq_w$ as in \prettyref{def:oriented-cube}.

\begin{proof}

    Let $\qq_N := [-N,N]^d, \ \qq_{w,w} := \{p +w: p\in \qq_w\}, \ \qq_{N,w} := \{p +w: p\in \qq_N\}$. 
    \par Using the above we get
    \begin{align*}
        (1-e^{i\langle \omega,t\rangle})\sum_{s\in {\Gg(\qq_N\setminus \qq_w)}}e^{i\langle s,t\rangle} &= \sum_{s\in {\Gg(\qq_N\setminus \qq_w)}}e^{i\langle s,t\rangle} - \sum_{s\in {\Gg(\qq_N\setminus \qq_w)}}e^{i\langle s + w,t\rangle} \\
        &= \sum_{s\in {\Gg(\qq_N\setminus \qq_w)}}e^{i\langle s,t\rangle} - \sum_{s\in {\Gg(\qq_{N,w}\setminus \qq_{w,w})}}e^{i\langle s,t\rangle} 
    \end{align*}
    Now let $\mathcal{R} = \qq_N \cup \qq_{N,w}$. We have that $\qq_w\cup\qq_{w,w}\subset \mathcal{R}$ because $\qq_{w}\subset\qq_{N-\|w\|}\subset\mathcal{R}$ from the definition of $\qq_w$. From symmetry, a similar argument holds for $\qq_{w,w}$ thus $\qq_{w,w}\subset\mathcal{R}$. We have
    \begin{align*}
        \sum_{s\in {\Gg(\qq_N\setminus \qq_w)}}e^{i\langle s,t\rangle} &= \sum_{s\in {\Gg(\qq_N\setminus \mathcal{R})}}e^{i\langle s,t\rangle} + \sum_{s\in {\Gg(\mathcal{R}\setminus(\qq_w\cup\qq_{w,w}))}}e^{i\langle s,t\rangle} + \sum_{s\in {\Gg(\qq_{w,w}\setminus(\qq_w\cap\qq_{w,w}))}}e^{i\langle s,t\rangle},
    \end{align*}
    and
    \begin{align*}
        \sum_{s\in {\Gg(\qq_{N,w}\setminus \qq_{w,w})}}e^{i\langle s,t\rangle} &= \sum_{s\in {\Gg(\qq_{N,w}\setminus \mathcal{R})}}e^{i\langle s,t\rangle} + \sum_{s\in {\Gg(\mathcal{R}\setminus(\qq_w\cup\qq_{w,w}))}}e^{i\langle s,t\rangle} + \sum_{s\in {\Gg(\qq_{w}\setminus(\qq_w\cap\qq_{w,w}))}}e^{i\langle s,t\rangle}.
    \end{align*}
    Thus we obtain
    \begin{align*}
        (1-e^{i\langle \omega,t\rangle})\sum_{s\in {\Gg(\qq_N\setminus \qq_w)}}e^{i\langle s,t\rangle} &= \sum_{s\in {\Gg(\qq_N\setminus \mathcal{R})}}e^{i\langle s,t\rangle} - \sum_{s\in {\Gg(\qq_{N,w}\setminus \mathcal{R})}}e^{i\langle s,t\rangle} \\
        &+ \sum_{s\in {\Gg(\qq_{w,w}\setminus(\qq_w\cap\qq_{w,w}))}}e^{i\langle s,t\rangle} - \sum_{s\in {\Gg(\qq_{w}\setminus(\qq_w\cap\qq_{w,w}))}}e^{i\langle s,t\rangle}.
    \end{align*}
    Now let 
    \begin{align*}
        &E_1 = (\qq_N\setminus \mathcal{R})\cup(\qq_{N,w}\setminus \mathcal{R}) = (\qq_N\cup\qq_{N,w})\setminus \mathcal{R}, \\
        &E_2 = \big(\qq_{w,w}\setminus(\qq_w\cap\qq_{w,w})\big)\cup\big(\qq_{w}\setminus(\qq_w\cap\qq_{w,w})\big).
    \end{align*}
    First, we have $|\partial E| \leq 2|\partial \qq_N|$ ($d-1$-dimensional volume) and since $\partial \qq_N$ consists of $2d$ faces, each being a $(d-1)$-dimensional hypercube of side length $2N+1$ we get that $|\partial \qq_N| = 2d \cdot (2N+1)^{d-1}$. In addition, the distance from $\partial E$ to $\mathcal{R}$ is at most $\|w\|$. Thus, we can bound the $d$-dimensional volume as $|E| \leq 2d{(2N+1)}^{d-1} \cdot \|w\| \leq c \cdot N^{d-1}$.
    \par Second, $E_2$ is a union of two $d$-dimensional slices of width $\|w\|$, and its volume is 
    $$2|\qq_w|\big(2{N-\|w\| \over \sqrt{d}}\big)^{-1}\|w\| = 2\big(2{N-\|w\| \over \sqrt{d}}\big)^{d-1}\|w\| \leq 2c\cdot \big(2{N-\|w\| \over \sqrt{d}}\big)^{d-1}.  $$
    Finally, using the triangle inequality we obtain
    \begin{align*}
        \bigg|(1-e^{i\langle \omega,t\rangle})\sum_{s\in {\Gg(\qq_N\setminus \qq_w)}}e^{i\langle s,t\rangle}\bigg| &\leq \bigg|\sum_{s\in {\Gg(\qq_N\setminus \mathcal{R})}}e^{i\langle s,t\rangle}\bigg| + \bigg|\sum_{s\in {\Gg(\qq_{N,w}\setminus \mathcal{R})}}e^{i\langle s,t\rangle}\bigg| \\
        &+ \bigg|\sum_{s\in {\Gg(\qq_{w,w}\setminus(\qq_w\cap\qq_{w,w}))}}e^{i\langle s,t\rangle}\bigg| + \bigg|\sum_{s\in {\Gg(\qq_{w}\setminus(\qq_w\cap\qq_{w,w}))}}e^{i\langle s,t\rangle}\bigg| \\ &= \big|E_1\big|_{\ZZ^d} + \big|E_2\big|_{\ZZ^d} \leq 2^d(|E_1| + |E_2|) \leq c_4\cdot N^{d-1},
    \end{align*}
    for some constant $c_4$. 
 See Figure~\ref{fig:trigo-cancel} for an illustration of this proof.

\end{proof}
\begin{figure}
    \centering
    \includegraphics[width=0.5\linewidth]{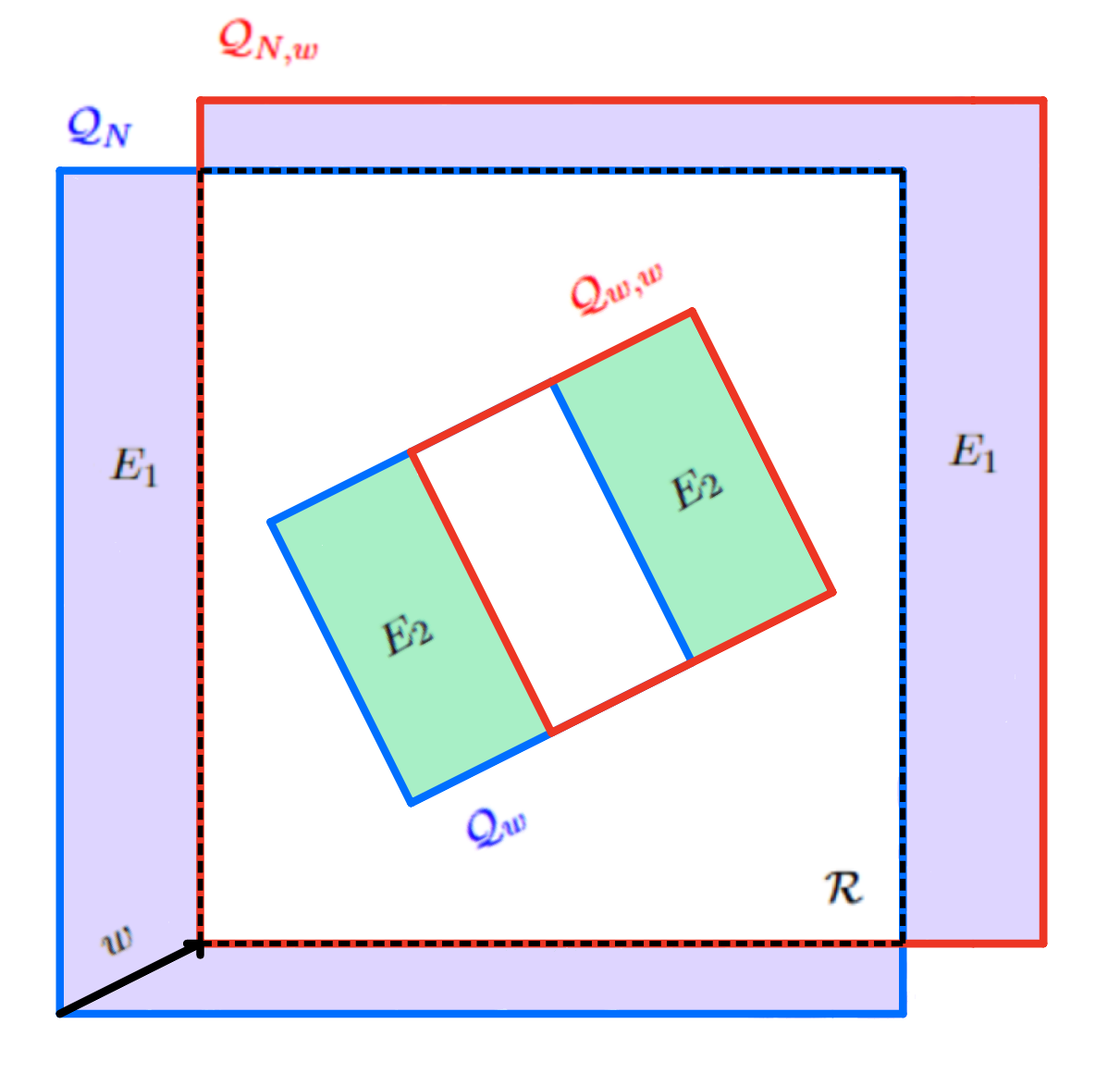}
    \caption{Illustration of Lemma \ref{lem:trig_cancel} for $d=2$. Blue cubes are $\qq_N$ and $\qq_w$. Red cubes are $\qq_{N,w}$ and $\qq_{w,w}$. Black vector is $w$. Purple region is $E_1$ and green region is $E_2$.}
    \label{fig:trigo-cancel}
\end{figure}

\section{Full spectral description}\label{spec-dis}
Here we establish the asymptotic scaling of all the singular values of $U_N(\ci):=U(\ci,\Gg_N)$.
Using the orthogonality result combined with the result on the singular values for a single cluster \cite{diab2024spectral}, we show that the singular values of $U_N$ are a scaled union of the singular values of the sub-matrices of $U(\mathcal{C}^{(j)},\Gg_N)$.

\par First, we recall the following lemma:
\begin{lemma}[Lemma 5.2, \cite{batenkov2021}]\label{lem:sv-sub-mat}
    Let $A\in \CC^{K\times n}, K \geq n$, such that $A$ is given in block form
    \[
        A = \big[A_1,\dots,A_M\big]
    \]
    with $A_j \in \CC^{K\times n^{(j)}}$ and $\sum_{j=1}^M n^{(j)}=n$. Let $\mathcal{L}_j \subset \CC^K$ be the subspace spanned by the columns of the sub-matrix $A_j$. Assume that for all $1\leq j,i \leq M, j\neq i$ and $0\leq\alpha\leq {1\over n}$,
    \[
        \measuredangle_{\min}(\Ll_j,\Ll_i) \geq {\pi \over 2} - \alpha.
    \]
    Moreover let
    \[ \sigma_1 \geq \dots \geq \sigma_n\]
    be the ordered collection of all the singular values of $A$, and let
    \[ \tilde{\sigma}_1 \geq \dots \geq \tilde{\sigma}_n\]
    be the ordered collection of all the singular values of the sub-matrices $\{A_j\}$. Then,
    \[
        (1-n\alpha)^{1\over 2}\tilde{\sigma}_j \leq \sigma_j \leq (1+n\alpha)^{1\over 2}\tilde{\sigma}_j, \ \ j=1,...,n.
    \]

\end{lemma}

\begin{proposition}\label{thm:multi-sv}
    Let $\mathcal{X}=\{x_1,...,x_n\} \subset \TT^d$ form an $((h^{(j)},n^{(j)})_{j=1}^M),\eta)$-clustered configuration where $\eta$ satisfies the condition in Proposition \ref{prop:proj}. Further, let 
    \[ \sigma_1 \geq \dots \geq \sigma_n\]
    be the singular values of $U_N(\ci)$ and
    \[ \tilde{\sigma}_1 \geq \dots \geq \tilde{\sigma}_n\]
    be all the singular values of sub-matrices $\{U_N(\mathcal{C}^{(j)})\}$. Putting $h = \max_j\{h^{(j)}\}_{j=1}^M$, then there exist positive constants $C_7, C_8$ depending only on $n^{(1)},...,n^{(M)}$ and $C_6,C_9,C_{10}$ depending also on the dimension $d$, such that for all $N$ satisfying $C_6 < N \leq {C_7 \over h}$ we have 
    \begin{equation}
        \big(1 - {C_9 \over N} - C_8Nh - {C_{10} \over N^{1 \over 2}}\big)^{1\over 2}\tilde{\sigma}_j \leq {\sigma}_j \leq \big(1 + {C_9 \over N} + C_8Nh + {C_{10} \over N^{1 \over 2}}\big)^{1\over 2}\tilde{\sigma}_j, \ \ j=1,...,M.
    \end{equation}
\end{proposition}
\begin{remark}
    In the above Proposition, we proved near-orthogonality using the "projection" technique for each pair of clusters in a configuration. That is, the admissible projection was different for each pair of clusters.
\end{remark}
\begin{proof}
    Applying Theorem \ref{thm:orth} on each pair of clusters $\mathcal{C}^{(j)}, \mathcal{C}^{(k)}, k\neq j$, we get that there exists an admissible projection $w \in \Gg_N\setminus \{\boldsymbol{0}\}$ as in Proposition \ref{prop:proj} and positive constants $C_6, C_7, C_{11},C_{12}$ and $C_{13}$ such that for all $N$ satisfying ${C_6} \leq N \leq {C_7 \over h}$ we have 
    \begin{equation}\label{eq:angle2clusters}
        \measuredangle_{\min}(\Ll(\mathcal{C}^{(j)},\Gg_N),\Ll(\mathcal{C}^{(k)},\Gg_N)) \geq {\pi \over 2} - \alpha, \ \alpha := {C_{11} \over N} + C_{12}Nh + {C_{13} \over N^{1 \over 2}} \leq {1\over n},
    \end{equation}
    where 
   \begin{align*}
    &C_{11} := \max_{1\leq j<k \leq M} {C_4(n^{(j)},n^{(k)})\over \eta_w}, \\
    &\eta_w := {1\over 16}{\sqrt{2}\pi \over (\sqrt{d}+1)} \\
    &C_{12} := \max_{1\leq j<k \leq M} C_3(n^{(j)},n^{(k)}), \\
    &C_{13} := \max_{1\leq j<k \leq M} C_5(n^{(j)},n^{(k)}), \\ 
    &C_6 := \max \big(\max_{1\leq j<k \leq M} C_1(n^{(j)},n^{(k)}), 3nC_{11}, 9n^2C_{13}^2\big), \\
    &C_7 := \max \big(\min_{1\leq j<k \leq M} C_2(n^{(j)},n^{(k)}), \frac{1}{3nC_{12}}\big).
\end{align*}
  Now we assume, without loss of generality that $U_N(\ci)$ is organized in block form, according to the clusters, as follows:
  \[
        U_N(\ci) = \big[U_N(\mathcal{C}^{(1)}),..., U_N(\mathcal{C}^{(M)})\big].
  \]
  By the above, we have the estimate \eqref{eq:angle2clusters} for ${C_6} \leq N \leq {C_7 \over h}$ when $\alpha \leq {1\over n}$.
  Using Lemma~\ref{lem:sv-sub-mat} with $A = U_N(\ci), A_j = U_N(\mathcal{C}^{(j)})$ we get that
  \[
        (1-n\alpha)^{1\over 2}\tilde{\sigma}_j \leq \sigma_j \leq (1+n\alpha)^{1\over 2}\tilde{\sigma}_j, \ \ j=1,...,n,
  \]
  obtaining the desired result with $C_6, C_7$ as above and $C_8=C_{12}n, C_9=C_{11}n$ and $C_{10}=C_{13}n$.
\end{proof}

Combining Theorems \ref{thm:sr} and \ref{thm:multi-sv}, we get the main result of this paper, Theorem \ref{thm:entire-spec}.

\subsection{Numerical validation}
We validate Theorem \ref{thm:entire-spec} for $d=2$. We give three different node configurations (with different cluster geometry) and show the scaling of the singular values of the corresponding Vandermonde matrix \eqref{eq:multiVan} as a function of $N\Delta = {1\over \srf}$, where $\Delta$ is the minimal separation between the nodes. The plots are in log-log scale.
\par All experiments include configuration of two clusters $\mathcal{C}_1$ and $\mathcal{C}_2$ separated by $\eta:=2$. We denote by $n$ the total number of nodes and by $\ell_1,\ell_2$ the number of nodes in each cluster. Also the number of grid samples is fixed (i.e $N$ is fixed) and $\Delta$ consists of the values $10^t$, where $t$ ranges from $-1$ to $-10$ in steps of $-{1\over 5}$.
For each experiment, we have:
\begin{enumerate}
    \item $N = 10, n=6, \ell_1=\ell_2=3$, both $\mathcal{C}_1$ and $\mathcal{C}_2$ in general position (see \prettyref{fig:random}).
    \item $N = 5, n=5, \ell_1=3, \ell_2=2$, $\mathcal{C}_1$ lies on a line and $\mathcal{C}_2$ is in general position (see \prettyref{fig:line}).
    \item $N = 10, n=8, \ell_1=6, \ell_2=2$, $\mathcal{C}_1$ lies on a parabola and $\mathcal{C}_2$ is in general position (see \prettyref{fig:parabola}).
\end{enumerate}
For experiment 2 (\prettyref{fig:line}), we have that the nodes $\ci:=\{(x_i,y_i)\}_{i=1}^3$ of $\mathcal{C}_1$ are on the line $y_i=x_i$ and $\mathcal{C}_1$ are in general position. According to \prettyref{thm:entire-spec}, we should see that the singular values of $U_N(\ci)$ scale like the singular values of $U_N(\mathcal{C}_1)$ and $U_N(\mathcal{C}_2)$. Using Theorem \ref{thm:sr}, for the first cluster we have $\rank(V_{\leq 2}(\mathcal{C}_1))=3$, where
        \begin{align*} V_{\leq 3}(\mathcal{C}_1) = \left[\begin{array}{c|cc|ccc}
                                1 & x_1 & y_1 & x_1^2 & x_1y_1 & y_1^2 \\
                                1 & x_2 & y_2 & x_2^2 & x_2y_2 & y_2^2 \\
                                1 & x_3 & y_3 & x_3^2 & x_3y_3 & y_3^2 
                        \end{array}\right] =  \left[\begin{array}{c|cc|ccc}
                                1 & x_1 & x_1 & x_1^2 & x_1x_1 & x_1^2 \\
                                1 & x_2 & x_2 & x_2^2 & x_2x_2 & x_2^2 \\
                                1 & x_3 & x_3 & x_3^2 & x_3x_3 & x_3^2 
                        \end{array}\right] = \begin{bmatrix}
                            V_0 & | & V_1 & | & V_2
                        \end{bmatrix}.
        \end{align*}
        We get that
        \begin{align*}
            t_0&=\rank(V_{\leq 0})=1,\\
             t_1&=\rank(V_{\leq 1})-\rank(V_{\leq 0})=2-1=1,\\ t_2&=\rank(V_{\leq 2})-\rank(V_{\leq 1})=2-1=1.
        \end{align*} Thus we will have that three of the singular values will scale like $(N\Delta)^0,(N\Delta)^1,(N\Delta)^2$ each.
        For the second cluster we have $\rank(V_{\leq 1}(\mathcal{C}_2))=3$, where
        \begin{align*} V_{\leq 3}(\mathcal{C}_1) = \left[\begin{array}{c|cc}
                                1 & x_1 & y_1 \\
                                1 & x_2 & y_2 \\
                                1 & x_3 & y_3 
                        \end{array}\right] = \begin{bmatrix}
                            V_0 & | & V_1
                        \end{bmatrix}.
        \end{align*}
        We get that
        \begin{align*}
            t_0&=\rank(V_{\leq 0})=1,\\
             t_1&=\rank(V_{\leq 1})-\rank(V_{\leq 0})=3-1=2.
        \end{align*} Thus we will have that two of the singular values will scale like $(N\Delta)^1,(N\Delta)^1$ and one will scale like $(N\Delta)^0$. We leave the details of the other experiments for the reader, as they are identical to the evaluations done for experiment 2.
\begin{figure}[htbp]
    \centering
    \begin{subfigure}[b]{0.43\textwidth}
        \centering
        \includegraphics[width=\textwidth]{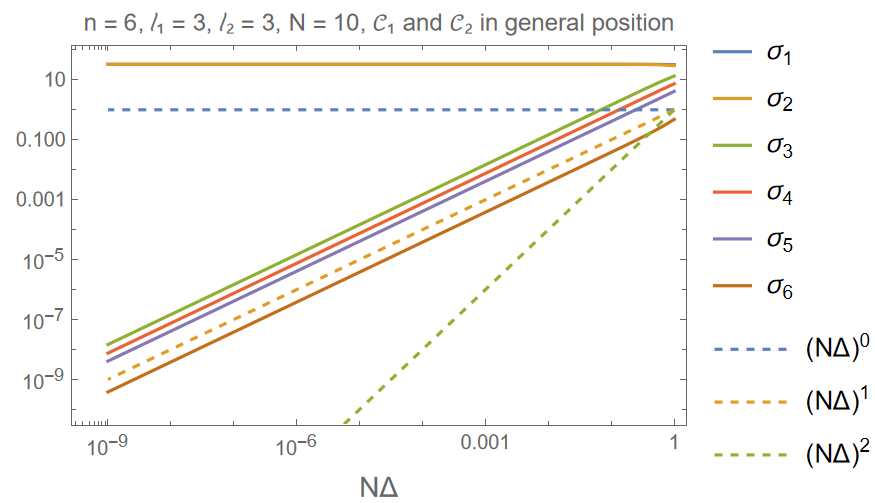}
        \caption{Two singular values scale like $(N\Delta)^0$ and four scale like $(N\Delta)^1$.}
        \label{fig:random}
    \end{subfigure}
    \hfill
    \begin{subfigure}[b]{0.47\textwidth}
        \centering
        \includegraphics[width=\textwidth]{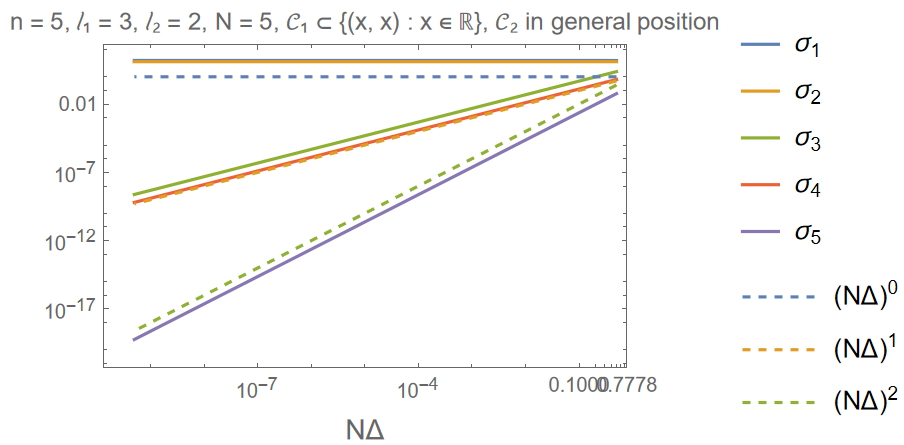}
        \caption{Two singular values scale like $(N\Delta)^0$, two scale like $(N\Delta)^1$ and one scales like $(N\Delta)^2$.}
        \label{fig:line}
    \end{subfigure}
    
    \vspace{0.5cm} 
    
    \begin{subfigure}[b]{0.60\textwidth}
        \centering
        \includegraphics[width=\textwidth]{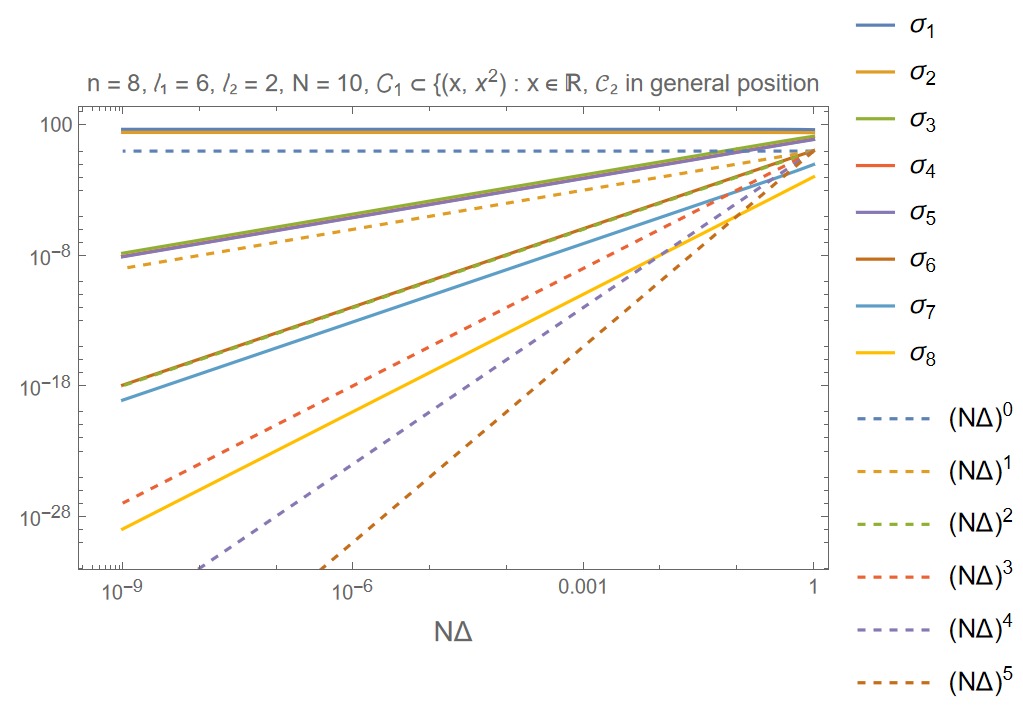}
        \caption{Two singular values scale like $(N\Delta)^0$, three scale like $(N\Delta)^1$, two scale like $(N\Delta)^2$ and one scales like $(N\Delta)^3$.}
        \label{fig:parabola}
    \end{subfigure}
    
    \caption{Singular values of $U_N(\mathcal{C}_1\cup\mathcal{C}_2)$ as a function of $N\Delta$}
    \label{fig:sv-van}
\end{figure}
These results confirm Theorem $\ref{thm:entire-spec}$.

\section{On-grid Min-max error bounds}\label{sec:min-max}

\par This section is organized as follows: we start with some definitions, then introduce the main result of on-grid min-max error bounds (in Theorem \ref{theorem:min-max-error}), then we provide some numerical results showing that 2D-ESPRIT attains the optimal bounds for signals with nodes forming multiple clusters in arbitrary geometry and lastly provide a discussion on min-max error bounds for nodes in general position (nodes with minimal discrete moment order).

\subsection{Preliminaries}
\begin{definition}\label{grid}
    For $\Delta > 0$ let $\Zz = \lfloor {\pi \over \Delta}\rfloor$ and denote by $\Tt_{\Delta}$ the discrete $d$-dimensional grid
    \begin{equation}
        \Tt_{\Delta} := \big\{\boldsymbol{k}\Delta \ : \ \boldsymbol{k}\in \Gg_{\Zz}\big\} \subset \TT^d.
    \end{equation}
    Also define $G:=G(\Delta) = |\Tt| = (2\Zz + 1)^d$.
\end{definition}

\begin{definition}\label{def:config-pro}
    For $h,\eta > 0$ and reference points $\{r_j\}_{j=1}^M \subset \TT^d$, we call a set of nodes $\{x_1,\dots,x_n\}\subset\TT^d$  a  {\bfseries $\mathbf{\big(h,\boldsymbol{\eta},r_1,\dots,r_M\big)}$-configuration profile}, denoted as $\mathcal{P}_h$, if
    \begin{enumerate}
        \item $\mathcal{P}_{h}$ forms a $((\mathcal{C}_j)_{j=1}^M,\eta)$- clustered configuration, where $\mathcal{C}_j$ is a $(h_j,\tau_j,n_j)$-cluster for some $0 < \tau_j \leq 1$ and $0 < h_j\leq h$.
        \item For any $j=1,\dots,M$ we have that for all $x\in\mathcal{C}_j$ we have $\|x-r_j\|_{\TT^d}\leq h$. We say that $r_j$ represents the cluster $\mathcal{C}_j$.
    \end{enumerate}
\end{definition}

\begin{remark}
    For any cluster $\mathcal{C}_j$ in a configuration profile $\mathcal{P}_{h}$, there exists $r_j' \neq r_j$ such that $r_j'$ represents $\mathcal{C}_j$. In fact, there are infinite other representatives.
\end{remark}

\begin{definition}\label{def:prior-con-pro}
     For $\eta,\ell$ as in definition \ref{def:multi_cluster} and $\Delta > 0$ let $\Rr := \Rr\big(\Delta,\eta,\ell,\{r_j\}_{j=1}^M\big)$ be the set of point distributions of the form $f = \sum_{j=1}^na_j\delta_{t_j}$ where 
     \begin{enumerate}
         \item For some constant $\nu \geq 1$ independent of $\Delta$, $\{t_j\}_{j=1}^n$ is a $\big(\nu\Delta,\eta,r_1,\dots,r_M\big)$-configuration profile as in definition~\ref{def:config-pro}.
         \item Every cluster $\mathcal{C}_j \subseteq \{t_j\}_{j=1}^n$ has at most $\ell$ nodes.
         \item $\{t_j\}_{j=1}^n \subset \Tt_{\Delta}$.
     \end{enumerate}
\end{definition}

\begin{definition}\label{def:order}
    Let $o : \Gg_{\Zz} \rightarrow \{1,2,\dots,|\Gg_{\Zz}|\}$ be a bijective function that defines an order on the set $\Gg_{\Zz}$.
\end{definition}
\begin{definition}
    For fixed $N \in \NN, \varepsilon > 0, N\Delta < 1$ and $f \in \Rr\big(\Delta,\eta,\ell,\{r_j\}_{j=1}^M\big)$, let
    \begin{equation}
        B_N^{\varepsilon}(f) = \bigg\{y \in \CC^{(2N+1)^d}\ : \ \bigg(\frac{1}{(2N+1)^d}\sum_{\boldsymbol{k}\in \Gg_N}\big|y_{o(\boldsymbol{k})} - \hat{f}(\boldsymbol{k})\big|^2\bigg)^{{1\over 2}} < \varepsilon\bigg\},
    \end{equation}
    where $\hat{f}$ are the Fourier coefficients as defined in \eqref{eq:samples} and $o(\cdot)$ is a function as defined in Definition~\ref{def:order}.
\end{definition}
\begin{definition}
    Let $\mathcal{A} := \mathcal{A}(\Rr,\Gg_N,\varepsilon)$ be the set of functions $\varphi$ that maps each $y \in \bigcup_{f \in \Rr}B_N^{\varepsilon}$ to a discrete distribution $\varphi(y) \in \Rr\big(\Delta,\eta,\ell,\{r_j\}_{j=1}^M\big)$.
\end{definition}
\begin{definition}
    For $f = \sum_{j=1}^na_j\delta_{t_j}$, the norm $\|f\|_2$ is the discrete $\ell_2$ norm of the coefficients vector:
    \begin{equation}
        \|f\|_2 := \bigg(\sum_{j=1}^n|a_j|^2\bigg)^{1\over 2}.
    \end{equation}
\end{definition}
\begin{definition}\label{def:min-max-hd}
    The $\ell^2$ min-max error for the on-the-grid model is 
    \begin{equation}
        \EE(\Rr,\Gg_N,\varepsilon) = \inf_{\varphi\in \mathcal{A}}\sup_{f \in \Rr}\sup_{y\in B^{\varepsilon}_N(f)} \| \varphi_y - f \|_2,
    \end{equation}
    where $\varphi_y := \varphi(y) \in \Rr$.
\end{definition}

Put $\Rr:=\Rr\big(\Delta,\eta,\ell,\{r_j\}_{j=1}^M\big)$. Then every discrete distribution $f=\sum_{j=1}^{n}a_j\delta_{t_j} \in \mathcal{R}$ can be identified with a sparse vector $x_{f} \in \mathbb{C}^{G}_{\mathcal{R}}\subset \mathbb{C}^{G}$, where $G=G(\Delta)$ from definition \ref{grid}, $\|x\|_0=\#(i|x_i \neq 0)$,
	\begin{equation}\label{xMu}
		(x_{f})_i := \begin{cases} 
			a_j & t_j= \boldsymbol{k}\Delta \in \mathbf{t} \ \land \ 1 \leq o(\boldsymbol{k}) \leq G\\
			0 & \text{otherwise} 
		\end{cases}
	\end{equation}
	for $j = 1,\dots, n$ and
	\[\mathbb{C}^{G}_{\Rr}:=\bigg\{x_{\gamma}: \gamma \in \Rr, x_{\gamma} \in \mathbb{C}^{G}, \|x_{\gamma}\|_0 \leq n \ \text{and $x_{\gamma}$ as defined in $\eqref{xMu}$} \bigg\} \ .\]

	A direct computation shows that for every $\omega \in \mathbb{R}$
	\[ \hat{f}(\omega) = \sum_{j=1}^{n}a_je^{i\langle\omega , t_j\rangle}.\]
	Thus we can write
	\[ \begin{pmatrix}
		\hat{f}(o^{-1}(1)) \\
		\hat{f}(o^{-1}(2)) \\
		\vdots \\
		\hat{f}(o^{-1}(G))
	\end{pmatrix}_{(G \times 1)} = \frac{1}{\sqrt{G}}\begin{pmatrix}
		\mathcal{F}_G
	\end{pmatrix} x_{f},
	\]
	where $\mathcal{F}_G = \big[e^{i\langle k,j\Delta\rangle}\big]_{k\in \Gg_{\Zz}}^{j\in \Gg_{\Zz}}$ is the $G\times G$ Discrete Fourier matrix.

	\begin{corollary}
		Assume that $(2N+1)^d \leq G$, let $\mathcal{F}_{\Gg_N} := \big[e^{i\langle k,j\Delta\rangle}\big]_{k\in \Gg_{N}}^{j\in \Gg_{\Zz}}$ be the sub-matrix of the $G \times G$ DFT matrix $\mathcal{F}_G$ and let $\tilde{\mathcal{F}}_{\Gg_N} = \frac{1}{\sqrt{(2N+1)^d}}\big(
			\mathcal{F}_{\Gg_N}\big)$. Then,
		\[ \tilde{\mathcal{F}}_{\Gg_N}x_{f} = U_{N}w_{f} \ , \ w_{f} =
		\begin{pmatrix}
			a_1 \\
			\vdots \\
			a_s 
		\end{pmatrix} \ . \]
		
	\end{corollary}
To establish an upper bound on the min-max error rates, we need to analyze the effect of subtracting two cluster configurations. Similarly to the one-dimensional case \cite{batenkov2021a}, subtracting two cluster configurations results in a larger cluster supported on the union of the original ones. We prove this in Lemma~\ref{apx:clusters-diff2}. Moreover, in Lemma~\ref{lem:line-mu}, we prove that the maximal attainable discrete moment order is when the resulting configuration has at most $2\ell$ nodes on a line, leading to the estimate in Theorem \ref{theorem:min-max-error}.
   \begin{lemma}\label{apx:clusters-diff2}
    Fix $\ell,\{r_j\}_{j=1}^M$, and assume that $\frac{2\pi}{a} \leq \eta \leq 2\pi$ for some constant $a>0$. Then for any $f_1, f_2 \in \Rr\big(\Delta,\eta,\ell,\{r_j\}_{j=1}^M\big)$, we have
    \[
        f_1 - f_2 \in \Rr\big(\Delta,\eta',\ell',\{r_j\}_{j=1}^M\big)
    \]
    for some $\ell' \leq 2\ell$, and $\eta'$ satisfying
    \[
        \frac{\pi}{a} \leq \eta' \leq \eta. 
    \]
    
\end{lemma}
\begin{proof}
    Let $T_k:=\supp(f_k), \ k=1,2$. Then we have that each $T_k$ is a $\big(\nu_k\Delta,\eta,r_1,\dots,r_M\big)$-configuration profile. For any $t_j \in \supp(f_1)$ define the following sets of distances:
    \[
        \mathcal{S}_j := \max\{\|y - t_j\|_{\TT^d} : y\in T_2\} \cap [0,{\eta \over 2}], \ j = 1,\dots, n.
    \]
    We get that there exists $\nu' \leq \nu_1+\nu_2$ such that $\max_j \{\mathcal{S}_j\} = \nu'\Delta$. Since $\nu'\Delta \leq {\eta\over 2}$, we get $\eta' = \eta - \nu'\Delta \geq \eta - {\eta \over 2} = {\eta \over 2} \geq {\pi \over a}$. We note that each cluster in $\supp(f_1 - f_2)$ has extent at most $\nu'\Delta$ (where $\nu'$ is independent of $\Delta$) and is the union of at most two clusters, one from $T_1$ and one from $T_2$, with the same reference point $r_j$.
\end{proof}

\begin{lemma}\label{lem:line-mu}
Let $\mathcal X=\{x_1,\dots,x_n\}\subset \mathbb R^d$ be a set of $n$ pairwise distinct points contained in an affine line. Then
\[
\mu(\mathcal X)=n-1.
\]
\end{lemma}

\begin{proof}
Since $\mathcal X$ is contained in an affine line, there exist $a,b\in\mathbb R^d$, with $b\neq 0$, and distinct scalars $t_1,\dots,t_n\in\mathbb R$ such that
\[
x_i=a+t_i b,\qquad i=1,\dots,n.
\]

For every multi-index $\alpha\in\mathbb Z_+^d$,
\[
    (a+t_ib)^\alpha, \ i=1,\dots,n,
\]
is a univariate polynomial of degree at most $|\alpha|$. Hence every column of $V_{\le k}(\mathcal X)$ is contained in the span of
\[
(1,\dots,1)^T,\ (t_i)_{i=1}^n,\ \dots,\ (t_i^k)_{i=1}^n,
\]
and therefore
\[
\rank(V_{\le k}(\mathcal X))\le k+1.
\]
In particular, if $k<n-1$, then
\[
\rank(V_{\le k}(\mathcal X))<n.
\]

On the other hand, the vectors
\[
(1,\dots,1)^T,\ (t_i)_{i=1}^n,\ \dots,\ (t_i^{\,n-1})_{i=1}^n
\]
belong to the column space of $V_{\le n-1}(\mathcal X)$, and form the columns of the univariate Vandermonde matrix
\[
\big[t_i^{\,j-1}\big]_{i,j=1}^n.
\]
Since for $i=1,\dots,n$, $t_i$ are pairwise distinct, this matrix is nonsingular. Hence
\[
\rank(V_{\le n-1}(\mathcal X))=n.
\]

Therefore,
\[
\mu(\mathcal X)=n-1.
\]
\end{proof}

\subsection{Min-max theorem and proofs}
    In our setting, we assume that the spike locations are restricted to a discrete grid of step size $\Delta$ and clustered around reference points $\{r_j\}_{j=1}^M$. In effect, for spikes with at most $\ell$ clustered nodes (and any geometry within the clusters), the following result shows that $\mathcal{E} \asymp SRF^{2\ell-1}\epsilon$ as $SRF\to\infty$. 
	
	\begin{theorem}\label{theorem:min-max-error}
        Fix $\ell\geq 2,\{r_j\}_{j=1}^M \subset \Tt_{\Delta}$. For $\Delta > 0$ and let $a:=\left\lfloor {\sqrt[d-1]{4(\sqrt{d}+1)}\over 2}\right\rfloor + 1$. Then the following holds:
		\begin{enumerate}
			\item
			For any ${2\pi\over a} \leq \eta \leq 2\pi$, there exist constants $C_{14},C_{15}$ such that for every $0<\delta_0<C_{15}$, $\alpha \in (0,{1})$ there exists $N_0$ such that for all $N \geq N_0$ and $\max\{{\alpha\delta_0\over\Delta},C_{14}\} < N < {\delta_0\over \Delta}$, it holds that
			\[\mathcal{E}\bigg(\Rr\big(\Delta,\eta,\ell,\{r_j\}_{j=1}^M\big),N,\epsilon\bigg) \lesssim {SRF}^{2\ell-1}\epsilon,\] where $SRF := \frac{1}{N\Delta}$.
			\item
			For any ${\pi\over a} \leq \eta \leq 2\pi$, there exist constants $C_{16},C_{17}$ such that for every $0 < \delta_0 < C_{17}$, $\alpha \in (0,{1})$ there exists $N_0$ such that for all $N \geq N_0$ and $\max\{{\alpha\delta_0\over\Delta},C_{16}\} < N < {\delta_0\over \Delta}$, it holds that
			\[\mathcal{E}\bigg(\Rr\big(\Delta,\eta,\ell,\{r_j\}_{j=1}^M\big),N,\epsilon\bigg) \gtrsim {SRF}^{2\ell-1}\epsilon,\] where $SRF := \frac{1}{N\Delta}$.
		\end{enumerate}
	\end{theorem}
	
	The proof largely repeats the arguments from \cite{Vandermonde,Li&Liao2020,batenkov2023}, together with the above established bounds on $\sigma_{\min}(U_N)$.
	
	\subsubsection{Proof of the upper bound}
   
        Let $\epsilon >0$ and for any $y\in \CC^{{(2N+1)}^d}$, let
        \[
            \mathcal{B}(\epsilon,y) = \{\gamma \in \Rr : \|\tilde{\mathcal{F}}_{\Gg_N}x_{\gamma} - y\|_2 \leq \epsilon\}.
        \]
        Moreover, consider the oracle estimator $\varphi_{\epsilon} \in \mathcal{A}$ defined as 
        \[
            \varphi_{\epsilon}(y) = \begin{cases}
                \text{any element of } \mathcal{B}(\epsilon,y) & \text{if }\mathcal{B}(\epsilon,y)\neq \emptyset, \\
                \gamma_0 & \text{else}
            \end{cases}
        \]
        where $\gamma_0$ is an arbitrary element of $\Rr$. Now let $f \in \Rr$ and $y = \tilde{\mathcal{F}}_{\Gg_N}x_{f} + e$ where $\|e\|_2 \leq \epsilon$. Then by definition, we have $f \in \mathcal{B}(\epsilon,y)$. 
    
    Let $\varphi_y := \varphi_{\epsilon}(y)$, then we have:
	\[ \mathcal{E}(\Rr,\Gg_N,\epsilon) \leq \sup_{f \in \Rr}\sup_{y\in B^{\epsilon}_N(\mu)}\|\varphi_y - f\|_2.\]
	By Lemma \ref{apx:clusters-diff2} for any $x_{f} \in \mathbb{C}^{G}_{\Rr\big(\Delta,\eta,\ell,\{r_j\}_{j=1}^M\big)}$, we have
	\begin{align*}
		(\varphi_y-f) &\in \Rr\big(\Delta,\eta',\ell',\{r_j\}_{j=1}^M\big)\\
		\Rightarrow x_{\varphi_y} - x_{f} &\in \mathbb{C}^{G}_{\Rr\big(\Delta,\eta',\ell',\{r_j\}_{j=1}^M\big)}
	\end{align*}
	where $\ell' \leq 2\ell$ and ${\pi\over a}\leq \eta' < \eta$. 
    Therefore, by applying Theorem \ref{thm:entire-spec}, we have constants $C_{14},C_{15}$ where for every $0 < \delta_0 < C_{15}$, $\alpha \in (0,{1})$ there exists $N_0$ such that for all $N \geq N_0$, 
    \begin{equation}\label{eq:N-bounds}
        \max\biggl\{{\alpha\delta_0 \over \Delta}, C_{14}\biggl\} < N < {\delta_0 \over \Delta}
    \end{equation}
    we have:
	\begin{align*}
		\frac{2\epsilon}{\|\varphi_y - f\|_2} &\geq \frac{\bigg|\|\tilde{\mathcal{F}}_{\Gg_N}x_{\varphi_y}-y\|_2 - \|\tilde{\mathcal{F}}_{\Gg_N}x_{f}-y\|_2\bigg|}{\|\varphi_y - f\|_2} \geq \frac{\|\tilde{\mathcal{F}}_{\Gg_N}(x_{\varphi_y}-x_{f})\|_2}{\|\varphi_y - f\|_2} \\
		&= \frac{\|U_{N}(\boldsymbol{t})(w_{\varphi_y}-w_{f})\|_2}{\|w_{\varphi_y}-w_{f}\|_2} 
		\geq \sigma_{\min}(U_{N}(\boldsymbol{t}))  \geq C{(N\Delta)}^{\mu(\mathcal{C}_{\boldsymbol{t}})},
	\end{align*}
    where $\boldsymbol{t}:=\supp(\varphi_y-f)$ and $\mathcal{C}_{\boldsymbol{t}}$ is the largest cluster in $\boldsymbol{t}$. We have that $\mu(\mathcal{C}_{\boldsymbol{t}}) \leq 2\ell-1$, thus
    \begin{align*}
        \frac{2\epsilon}{\|\varphi_y - f\|_2} \geq C{(N\Delta)}^{\mu(\mathcal{C}_{\boldsymbol{t}})}  \geq C{(N\Delta)}^{2\ell-1}.
    \end{align*}

    \begin{remark}
    
        In our formulation of the min-max bound, we fix the configuration $\{r_j\}_{j=1}^M$. On the one hand, we believe this assumption to be a technicality specific to the on-grid setting which can probably be removed. On the other hand, it describes a ``best-case scenario'' where the reconstruction algorithm is already provided with the approximate clusters, and only needs to super-resolve the nodes in each cluster. Therefore, it is reasonable to assume that the exponent $2\ell-1$ cannot be substantially improved if the above assumption is removed.

\end{remark}
    
	\subsubsection{Proof of the lower bound}
	For some $\nu \geq 1$, pick any $(\nu\Delta,\eta,\{r_j\}_{j=1}^M)$-configuration profile $\boldsymbol{t}=(t_1,\dots,t_{2n})$ such that $\boldsymbol{t}=\supp{f}$ for some $f\in \Rr\big(\Delta,\eta,2\ell,\{r_j\}_{j=1}^M\big)$ and the nodes in each cluster belong to an affine one-dimensional subspace (i.e. they form a line). Let $w\in\mathbb{C}^{2n}$ be a unit norm singular vector corresponding to $\sigma_{\min}(U_N(\boldsymbol{t}))$ and define $f=\sum_{j=1}^{2n} (w)_{j} \delta_{t_j}  \in \Rr\big(\Delta,\eta,2\ell,\{r_j\}_{j=1}^M\big)$, so we actually have that $w=w_{f}$. By this construction, we obtain
	\[ \sigma:= \sigma_{\min}(U_{N}(\supp(f))= \|U_{N}w_{f}\|_2 = \|\tilde{\mathcal{F}}_{\Gg_n}x_{f}\|_2 \ . \]
	
	Write $\boldsymbol{t}$ as a disjoint union of two configuration profiles $\boldsymbol{t}_1,\boldsymbol{t}_2$, implying that $f=f_1-f_2$ where $\supp f_i = \boldsymbol{t}_i$ and $f_i\in \Rr\big(\Delta,\eta,\ell,\{r_j\}_{j=1}^M\big)$ for $i=1,2$. Let $x_i=\frac{\epsilon}{\sigma}x_{f_i} \in \mathbb{C}^{G}_{\Rr\big(\Delta,\eta,\ell,\{r_j\}_{j=1}^M\big)}$,  for $i=1,2$, so that $\frac{\epsilon}{\sigma}x_{f} = x_1-x_2$.
	
	Now suppose we are given the data:
	\[ y = \tilde{\mathcal{F}}_{\Gg_N}x_1 = \tilde{\mathcal{F}}_{\Gg_N}x_2 + \tilde{\mathcal{F}}_{\Gg_N}(x_1-x_2).\]
	Let $e:=\tilde{\mathcal{F}}_{\Gg_N}(x_1-x_2) \in \mathbb{C}^{(2N+1)^d}$.
	The previous equations imply:
	\[ \|e\|_2 = \|\tilde{\mathcal{F}}_{\Gg_N}(x_1-x_2)\|_2 = \frac{\epsilon}{\sigma}\|\tilde{\mathcal{F}}_{\Gg_N}x_{f}\|_2 = \epsilon.\]
	For an arbitrary $\varphi$ we have
	\begin{align*}
		\frac{\epsilon}{\sigma} &= \frac{\epsilon}{\sigma} \|w_{f}\|_2 = \frac{\epsilon}{\sigma} \|x_{f}\|_2\\
		&= \|x_1-x_2\|_2 \\
		& \leq \|x_1-x_{\varphi_{y}}\|_2 + \|x_2-x_{\varphi_{y}}\|_2\\
		& \leq 2\max_{k=1,2} \|x_k-x_{\varphi_{y}}\|_2
	\end{align*}
	and so by definition of $\mathcal{E}$ and Theorem \ref{thm:entire-spec} we conclude that for constants $C_{16},C_{17}$ and for every $ 0 < \delta_0 < C_{17}, \alpha \in (0,{1})$ there exists $N_0$ such that for all $N \geq N_0$ and $\max\{{\alpha\delta_0 \over \Delta}, {C_{16}}\} < N < {\delta_0 \over \Delta}$ we have
	\[ \Rightarrow \mathcal{E}(\Rr,\Gg_N,\epsilon) \geq \inf_{\varphi \in \mathcal{A}}\max_{k=1,2}\|x_{\varphi_y} - x_k\|_2 \geq \frac{\epsilon}{2\sigma} \geq \frac{\epsilon}{2 C{(N\Delta)}^{\mu(\mathcal{C}_{\boldsymbol{t}})}}, \]
    where $\mathcal{C}_{\boldsymbol{t}}$ is the largest cluster in $\boldsymbol{t}$ of size $2\ell$ and $\mu(\mathcal{C}_{\boldsymbol{t}}) = 2\ell-1$ from Lemma \ref{lem:line-mu}. \qed
\subsection{Bounds validation via 2D ESPRIT}\label{esprit}
\par One of the most popular signal subspace methods is the ESPRIT (Estimation of Signal Parameters via Rotational Invariance Techniques) method \cite{ESPRIT1990}. Its main idea is to exploit a shift-invariance property inherent in the Vandermonde structure of the signal. For more details and explanation on 2D-ESPRIT see section \ref{appendix:2d-esprit} in the appendix.
We implemented the 2D-ESPRIT algorithm according to \cite[Section IV-A]{sahnoun2017}. We consider the asymptotic regime where $N$ is fixed and the minimal separation between the nodes $\Delta \rightarrow 0$. In all experiments, there are two clusters $\mathcal{C}_1$ and $\mathcal{C}_2$ with the same number nodes $n_1=n_2$ that varied from $2$ to $5$. The clusters are separated by $2$, where $\mathcal{C}_1$ had different geometry in each experiment (either random, a line $y=7x$, or a parabola $y=7x^2$). Moreover, the parameters of ESPRIT were set as follows: $N_1=N_2=50,\; L_1=L_2=20,\;\beta_1=\beta_2=0.5$ and the noise level was set to $\sigma=10^{-10}$.

The results are shown in Figure \ref{fig:esprit-all}. We fitted the scaled error
$$
\mathcal{K}_{x_j}:= | \hat{x}_j^{ESPRIT} - x_j | \cdot N / \sigma
$$
to the curve $\mathcal{K}_{x_j} \sim \Delta^{m}$. It can be seen that the accuracy crucially depends on the geometry, providing best scaling for the random configuration, and worst scaling for the single-line configuration. Comparing these result to the min-max error bounds established in  Theorem~\ref{theorem:min-max-error}, we see that ESPRIT achieves the optimal scaling for nodes on line geometry (see \prettyref{fig:esprit-line}) and achieves better scaling for the other geometries (see figures \ref{fig:esprit-random} and \ref{fig:esprit-parabola}). Thus, using this numerical observation, we conclude that 2D-ESPRIT is likely to be optimal.

\begin{figure}[htbp]
    \centering
    \begin{subfigure}[b]{0.45\textwidth}
        \centering
        \includegraphics[width=\textwidth]{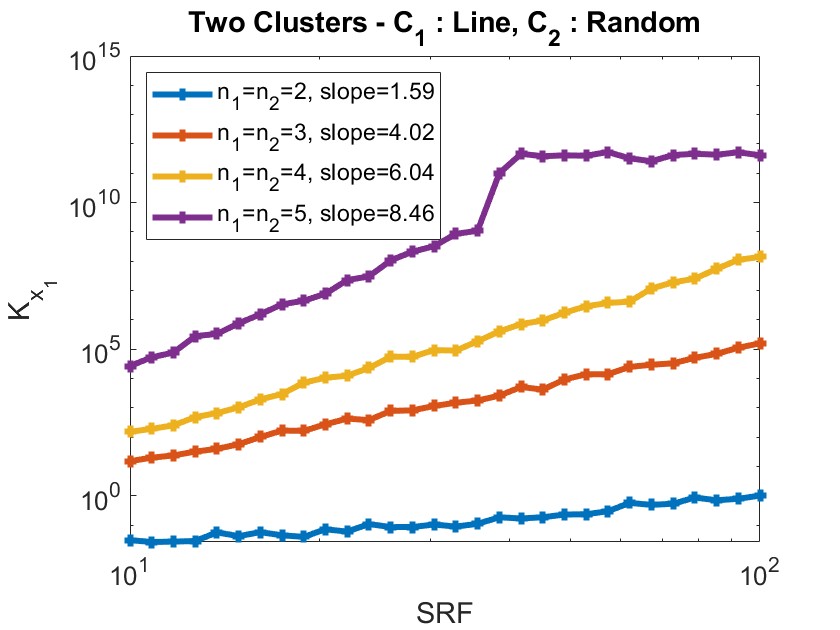}
        \caption{ }
        \label{fig:esprit-line}
    \end{subfigure}
    \hfill
    \begin{subfigure}[b]{0.45\textwidth}
        \centering
        \includegraphics[width=\textwidth]{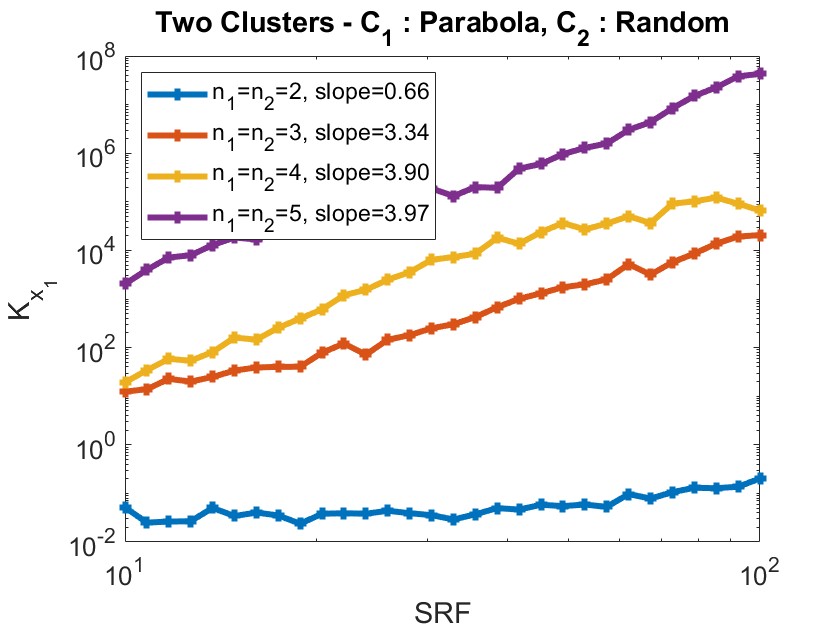}
        \caption{ }
        \label{fig:esprit-parabola}
    \end{subfigure}
    
    \vspace{0.5cm}
    
    \begin{subfigure}[b]{0.50\textwidth}
        \centering
        \includegraphics[width=\textwidth]{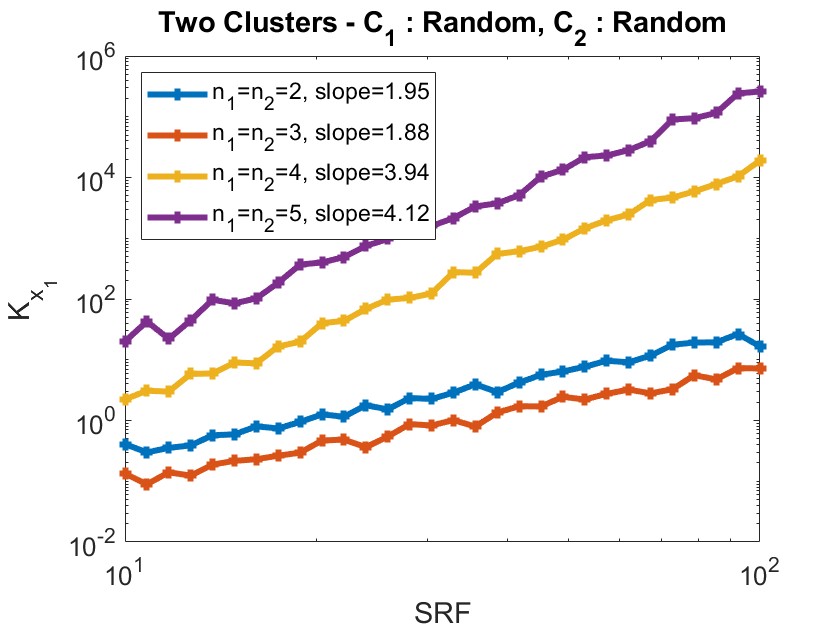}
        \caption{ }
        \label{fig:esprit-random}
    \end{subfigure}
    
    \caption{Accuracy of first node $x_1$ using 2D ESPRIT for different geometries.}
    \label{fig:esprit-all}
\end{figure}

\subsection{Towards min-max error bounds for nodes in general position}\label{subsec:min-max-random}
\par We proved (in Theorem \ref{theorem:min-max-error}) that the "on-grid" min-max error bounds scale like $\mathcal{E}(\mathcal{R},N,\mathcal{\epsilon}) \asymp \srf^{2\ell - 1}\epsilon$. However, the exponent can be smaller than $2\ell-1$. For example, let $\mathcal{X}_1 = \{(0, 0), (\sqrt{2}, 2)\}$ and $\mathcal{X}_2 = \{(1, 1), (2, 4)\}$ represent the node configurations of the measures $\mu_1$ and $\mu_2$, respectively, where the nodes are aligned along a line in each configuration. Then, for the measure $\mu := \mu_1 - \mu_2$, the node configuration becomes $\mathcal{X} = \{(0, 0), (1, 1), (\sqrt{2}, 2), (2, 4)\}$, with the nodes now lying on a parabola. Therefore, the discrete moment order of the union is $2$, which is smaller than $2\ell-1=3$, resulting in a better scaling in Theorem~\ref{theorem:min-max-error}. This phenomenon should be further investigated, since it suggests that any prior knowledge on the geometry of the unknown nodes could affect the scaling and give better stability bounds.


\par In the multi-dimensional case, the nodes are most likely to be in a general position rather than on an algebraic variety. In this context, general position means that the discrete moment order is minimal possible, i.e.\ for a cluster of size $2\ell$ it equals the unique integer $m$ such that ${d + m - 1 \choose d} < 2\ell \leq {d + m \choose d}$. If we restrict the class $\mathcal{R}$ to be all the signals of the form \eqref{eq:spike} with nodes in general position only, we expect the min-max error bounds to be much better than ${\srf^{2\ell - 1}}\epsilon$ and closer to ${\srf^{m}}\epsilon$. Note that such $m$ is the smallest discrete moment order for any cluster of size $2\ell$, since smaller $m$ wouldn't satisfy the condition in definition \ref{def:mu-def}. For comparison, if the nodes are on a line we would have that $m = 2\ell - 1$. We cannot expect the bound ${\srf^{m}}\epsilon$ itself, since the min-max error depends on the difference between two configurations, and the difference between two clusters in general position may not be in general position. If true, this seems to suggest that multidimensional super-resolution may be considerably more stable than in the 1D case, at least for sources in general position. We leave this for future work.

\appendix
\section{2D-ESPRIT}\label{appendix:2d-esprit}

Here we describe the extension of ESPRIT to the two-dimensional setting as in \cite[Section IV-A]{sahnoun2017}. Consider the model \eqref{eq:spike} with $d=2$, i.e.,
\[
f(t) = \sum_{j=1}^n \alpha_j \delta(t - x_j), \quad x_j = (x_{j,1},x_{j,2}) \in \TT^2,
\]
and suppose we are given noisy Fourier measurements on a rectangular grid
\[
\Gg_{N_1,N_2} := \{-N_1,\dots,N_1\} \times \{-N_2,\dots,N_2\}.
\]
Then
\[
\hat{f}(\omega_1,\omega_2)
=
\sum_{j=1}^n \alpha_j e^{\imath(\omega_1 x_{j,1} + \omega_2 x_{j,2})}
+
\epsilon(\omega_1,\omega_2).
\]

To expose the shift-invariance structure, we rearrange the data into a structured matrix. Fix $L_1,L_2$ with $1 \leq L_k \leq 2N_k+1$, and define $K_k := 2N_k+1 - L_k + 1$. For each $(p,q)$ with $1 \leq p \leq K_1$ and $1 \leq q \leq K_2$, consider the block of measurements
\[
Y^{(p,q)} =
\begin{bmatrix}
\hat{f}(p-1,q-1) & \hat{f}(p-1,q) & \cdots & \hat{f}(p-1,q+L_2-2) \\
\hat{f}(p,q-1) & \hat{f}(p,q) & \cdots & \hat{f}(p,q+L_2-2) \\
\vdots & \vdots & & \vdots \\
\hat{f}(p+L_1-2,q-1) & \hat{f}(p+L_1-2,q) & \cdots & \hat{f}(p+L_1-2,q+L_2-2)
\end{bmatrix}.
\]

We define the row-major vectorization $\mathrm{vec}(\cdot)$ by
\[
\mathrm{vec}(Y) :=
\begin{bmatrix}
Y_{1,1} & Y_{1,2} & \cdots & Y_{1,L_2} & Y_{2,1} & \cdots & Y_{L_1,L_2}
\end{bmatrix}^T \in \CC^{L_1L_2}.
\]

The multilevel Hankel matrix is then given by
\[
H =
\begin{bmatrix}
\mathrm{vec}(Y^{(1,1)}) & \mathrm{vec}(Y^{(1,2)}) & \cdots & \mathrm{vec}(Y^{(K_1,K_2)})
\end{bmatrix}
\in \CC^{(L_1L_2)\times (K_1K_2)}.
\]

In the noiseless case, $H$ admits a factorization of the form
\[
H = P \diag(\alpha) Q^T,
\]
where $\alpha = (\alpha_1,\dots,\alpha_n)^T$, and
\[
P = \mathbf{V}_{L_1}(\{x_{j,1}\}_{j=1}^n)\ \odot\ \mathbf{V}_{L_2}(\{x_{j,2}\}_{j=1}^n),
\]
\[
Q = \mathbf{V}_{K_1}(\{x_{j,1}\}_{j=1}^n)\ \odot\ \mathbf{V}_{K_2}(\{x_{j,2}\}_{j=1}^n),
\]
with $\odot$ denoting the Khatri--Rao (column-wise Kronecker) product. 
In particular, $\rank(H)=n$ in the noiseless case.

Let
\[
H = U \Sigma V^*
\]
be the singular value decomposition, and let $U_s \in \CC^{(L_1L_2)\times n}$ denote the matrix of the $n$ leading left singular vectors. Then, we have that $\mathrm{span}(U_s) = \mathrm{span}(P)$.

To recover the nodes, we exploit shift-invariance along each coordinate direction. Recall that the rows of $H$, and hence of $U_s \in \CC^{(L_1L_2)\times n}$, are indexed by pairs $(\ell_1,\ell_2)$ with
\[
0 \leq \ell_1 \leq L_1-1, \quad 0 \leq \ell_2 \leq L_2-1,
\]
according to the row-major vectorization. 

We define submatrices by selecting rows corresponding to shifted index sets. For the first dimension, let
\[
\mathcal{I}_1^- = \{(\ell_1,\ell_2) : 0 \leq \ell_1 \leq L_1-2,\ 0 \leq \ell_2 \leq L_2-1\},
\]
\[
\mathcal{I}_1^+ = \{(\ell_1,\ell_2) : 1 \leq \ell_1 \leq L_1-1,\ 0 \leq \ell_2 \leq L_2-1\}.
\]
We then define $U_s^{(1,-)}$ and $U_{s,(1,-)}$ as the submatrices of $U_s$ formed by selecting rows indexed by $\mathcal{I}_1^-$ and $\mathcal{I}_1^+$, respectively.

Similarly, for the second dimension, define
\[
\mathcal{I}_2^- = \{(\ell_1,\ell_2) : 0 \leq \ell_1 \leq L_1-1,\ 0 \leq \ell_2 \leq L_2-2\},
\]
\[
\mathcal{I}_2^+ = \{(\ell_1,\ell_2) : 0 \leq \ell_1 \leq L_1-1,\ 1 \leq \ell_2 \leq L_2-1\},
\]
and let $U_s^{(2,-)}$ and $U_{s,(2,-)}$ be the corresponding submatrices.
Then there exist matrices $F_1,F_2 \in \CC^{n\times n}$ such that
\[
U_s^{(1,-)} F_1 = U_{s,(1,-)}, \quad
U_s^{(2,-)} F_2 = U_{s,(2,-)}.
\]
These are computed via
\[
F_1 = (U_s^{(1,-)})^\dagger U_{s,(1,-)}, \quad
F_2 = (U_s^{(2,-)})^\dagger U_{s,(2,-)}.
\]

In the noiseless case, $F_1$ and $F_2$ are diagonalizable, and their eigenvalues satisfy
\[
\mathrm{eig}(F_1) = \{e^{\imath x_{j,1}}\}_{j=1}^n, \quad
\mathrm{eig}(F_2) = \{e^{\imath x_{j,2}}\}_{j=1}^n.
\]
To correctly pair the coordinates, one forms
\[
K = \beta_1 F_1 + \beta_2 F_2,
\]
for some $\beta_1,\beta_2 \in \CC$, and computes its eigendecomposition
\[
K = T \Lambda T^{-1}.
\]
Then
\[
D_1 = T^{-1}F_1T, \quad D_2 = T^{-1}F_2T,
\]
are diagonal, and we have
\[
x_{j,1} = \arg((D_1)_{jj}), \quad x_{j,2} = \arg((D_2)_{jj}).
\]
Note that $\beta_1,\beta_2$ are generic coefficients chosen such that the eigenvalues of $K$ are distinct, ensuring that $K$ has a unique eigendecomposition and enabling consistent pairing of the coordinates across dimensions.

The shifts used above are canonical, but this is not necessary. Recall
that the rows of $U_s$ are indexed by $\ell \in \Lambda := \{0,\dots,L_1-1\}
\times \{0,\dots,L_2-1\}$, and that the $j$-th column of $P$ has entries
$e^{\imath\langle \ell,x_j\rangle}$. Hence, for any $m \in \ZZ^2$, we obtain a
shift-invariance relation between the rows of $U_s$ indexed by $\Lambda_m^- :=
\{\ell \in \Lambda : \ell + m \in \Lambda\}$ and those indexed by $\Lambda_m^-+m$.
The corresponding matrix $F_m$ has eigenvalues $\{e^{\imath\langle m,x_j\rangle}
\}_{j=1}^n$. In particular, any $d$ linearly independent shifts determine the
nodes, since the phases $\langle m_k,x_j\rangle$ give a linear system for $x_j$.

It is useful to take more than $d$ shifts when the measurements are noisy.
In this case the linear system is overdetermined and can be solved by least
squares, and the pairing step can be replaced by a joint diagonalization of the
matrices $F_{m_k}$, that is, by finding a single $T$ which makes all of them as close to
diagonal as possible. Note that the shifts cannot be taken too large, since the
phases are determined only modulo $2\pi$ and the number of rows in
$\Lambda_m^-$ decreases as $|m|$ grows.

Compared to the one-dimensional case, the multidimensional setting requires constructing a multilevel Hankel matrix, solving multiple coupled shift-invariance relations and additional pairing step. We describe the 2D-ESPRIT in algorithm \ref{2d-esprit}.

\begin{algorithm}
\caption{2D-ESPRIT}
\begin{algorithmic}[1]\label{2d-esprit}
\REQUIRE Measurements $\hat{f}(\omega_1,\omega_2)$ on $\Gg_{N_1,N_2}$, $L_1,L_2$, model order $n$, $\beta_1,\beta_2$.
\STATE Set $K_k = 2N_k+1 - L_k + 1$.
\STATE Construct the multilevel Hankel matrix $H$.
\STATE Compute SVD: $H = U\Sigma V^*$.
\STATE Set $U_s =$ first $n$ columns of $U$.
\STATE Form shifted matrices $U_s^{(1,-)}, U_{s,(1,-)}, U_s^{(2,-)}, U_{s,(2,-)}$.
\STATE Compute
\[
F_1 = (U_s^{(1,-)})^\dagger U_{s,(1,-)}, \quad
F_2 = (U_s^{(2,-)})^\dagger U_{s,(2,-)}.
\]
\STATE Form $K = \beta_1 F_1 + \beta_2 F_2$.
\STATE Compute eigendecomposition $K = T \Lambda T^{-1}$.
\STATE Compute $D_1 = T^{-1}F_1T$, $D_2 = T^{-1}F_2T$.
\STATE Extract nodes:
\[
x_{j,1} = \arg((D_1)_{jj}), \quad
x_{j,2} = \arg((D_2)_{jj}).
\]
\STATE Recover amplitudes via least squares.
\RETURN $\{x_j\}, \{\alpha_j\}$.
\end{algorithmic}
\end{algorithm}
\bibliographystyle{alpha}
\bibliography{ref}

\end{document}